\documentclass[11pt]{amsart}    
\usepackage{latexsym, amsmath, amsthm, amssymb, setspace, verbatim}
\usepackage[all]{xy}
\usepackage{longtable}
\usepackage[applemac]{inputenc}
\usepackage{hyperref}
\usepackage{color}
\usepackage[usenames,dvipsnames]{xcolor}

\title[Rokhlin actions on UHF absorbing \cstar-algebras]{ Rokhlin actions of finite groups on UHF-absorbing \cstar-algebras }
\author{ Sel\c{c}uk Barlak \and Gábor Szabó }
\address{Westfälische Wilhelms-Universität, Fachbereich Mathematik, \phantom{--------------}\linebreak \text{}\hspace{3.5mm} Einsteinstrasse 62, 48149 Münster, Germany}
\email{selcuk.barlak@uni-muenster.de, gabor.szabo@uni-muenster.de}
\thanks{\emph{Supported by:} SFB 878 \emph{Groups, Geometry and Actions} and GIF Grant 1137-30.6/2011}
\subjclass[2010]{46L55, 46L35}

\begin{document}

\renewcommand\matrix[1]{\left(\begin{array}{*{10}{c}} #1 \end{array}\right)}  
\newcommand\set[1]{\left\{#1\right\}}  
\newcommand\mset[1]{\left\{\!\!\left\{#1\right\}\!\!\right\}}

\newcommand{\IA}[0]{\mathbb{A}} \newcommand{\IB}[0]{\mathbb{B}}
\newcommand{\IC}[0]{\mathbb{C}} \newcommand{\ID}[0]{\mathbb{D}}
\newcommand{\IE}[0]{\mathbb{E}} \newcommand{\IF}[0]{\mathbb{F}}
\newcommand{\IG}[0]{\mathbb{G}} \newcommand{\IH}[0]{\mathbb{H}}
\newcommand{\II}[0]{\mathbb{I}} \renewcommand{\IJ}[0]{\mathbb{J}}
\newcommand{\IK}[0]{\mathbb{K}} \newcommand{\IL}[0]{\mathbb{L}}
\newcommand{\IM}[0]{\mathbb{M}} \newcommand{\IN}[0]{\mathbb{N}}
\newcommand{\IO}[0]{\mathbb{O}} \newcommand{\IP}[0]{\mathbb{P}}
\newcommand{\IQ}[0]{\mathbb{Q}} \newcommand{\IR}[0]{\mathbb{R}}
\newcommand{\IS}[0]{\mathbb{S}} \newcommand{\IT}[0]{\mathbb{T}}
\newcommand{\IU}[0]{\mathbb{U}} \newcommand{\IV}[0]{\mathbb{V}}
\newcommand{\IW}[0]{\mathbb{W}} \newcommand{\IX}[0]{\mathbb{X}}
\newcommand{\IY}[0]{\mathbb{Y}} \newcommand{\IZ}[0]{\mathbb{Z}}

\newcommand{\CA}[0]{\mathcal{A}} \newcommand{\CB}[0]{\mathcal{B}}
\newcommand{\CC}[0]{\mathcal{C}} \newcommand{\CD}[0]{\mathcal{D}}
\newcommand{\CE}[0]{\mathcal{E}} \newcommand{\CF}[0]{\mathcal{F}}
\newcommand{\CG}[0]{\mathcal{G}} \newcommand{\CH}[0]{\mathcal{H}}
\newcommand{\CI}[0]{\mathcal{I}} \newcommand{\CJ}[0]{\mathcal{J}}
\newcommand{\CK}[0]{\mathcal{K}} \newcommand{\CL}[0]{\mathcal{L}}
\newcommand{\CM}[0]{\mathcal{M}} \newcommand{\CN}[0]{\mathcal{N}}
\newcommand{\CO}[0]{\mathcal{O}} \newcommand{\CP}[0]{\mathcal{P}}
\newcommand{\CQ}[0]{\mathcal{Q}} \newcommand{\CR}[0]{\mathcal{R}}
\newcommand{\CS}[0]{\mathcal{S}} \newcommand{\CT}[0]{\mathcal{T}}
\newcommand{\CU}[0]{\mathcal{U}} \newcommand{\CV}[0]{\mathcal{V}}
\newcommand{\CW}[0]{\mathcal{W}} \newcommand{\CX}[0]{\mathcal{X}}
\newcommand{\CY}[0]{\mathcal{Y}} \newcommand{\CZ}[0]{\mathcal{Z}}

\newcommand{\FA}[0]{\mathfrak{A}} \newcommand{\FB}[0]{\mathfrak{B}}
\newcommand{\FC}[0]{\mathfrak{C}} \newcommand{\FD}[0]{\mathfrak{D}}
\newcommand{\FE}[0]{\mathfrak{E}} \newcommand{\FF}[0]{\mathfrak{F}}
\newcommand{\FG}[0]{\mathfrak{G}} \newcommand{\FH}[0]{\mathfrak{H}}
\newcommand{\FI}[0]{\mathfrak{I}} \newcommand{\FJ}[0]{\mathfrak{J}}
\newcommand{\FK}[0]{\mathfrak{K}} \newcommand{\FL}[0]{\mathfrak{L}}
\newcommand{\FM}[0]{\mathfrak{M}} \newcommand{\FN}[0]{\mathfrak{N}}
\newcommand{\FO}[0]{\mathfrak{O}} \newcommand{\FP}[0]{\mathfrak{P}}
\newcommand{\FQ}[0]{\mathfrak{Q}} \newcommand{\FR}[0]{\mathfrak{R}}
\newcommand{\FS}[0]{\mathfrak{S}} \newcommand{\FT}[0]{\mathfrak{T}}
\newcommand{\FU}[0]{\mathfrak{U}} \newcommand{\FV}[0]{\mathfrak{V}}
\newcommand{\FW}[0]{\mathfrak{W}} \newcommand{\FX}[0]{\mathfrak{X}}
\newcommand{\FY}[0]{\mathfrak{Y}} \newcommand{\FZ}[0]{\mathfrak{Z}}

\newcommand{\Ra}[0]{\Rightarrow}
\newcommand{\La}[0]{\Leftarrow}
\newcommand{\LRa}[0]{\Leftrightarrow}

\renewcommand{\phi}[0]{\varphi}
\newcommand{\eps}[0]{\varepsilon}

\newcommand{\quer}[0]{\overline}
\newcommand{\uber}[0]{\choose}
\newcommand{\ord}[0]{\operatorname{ord}}		
\newcommand{\GL}[0]{\operatorname{GL}}
\newcommand{\supp}[0]{\operatorname{supp}}	
\newcommand{\id}[0]{\operatorname{id}}		
\newcommand{\Sp}[0]{\operatorname{Sp}}		
\newcommand{\eins}[0]{\mathbf{1}}			
\newcommand{\diag}[0]{\operatorname{diag}}
\newcommand{\auf}[1]{\quad\stackrel{#1}{\longrightarrow}\quad}
\newcommand{\prim}[0]{\operatorname{Prim}}
\newcommand{\ad}[0]{\operatorname{Ad}}
\newcommand{\ext}[0]{\operatorname{Ext}}
\newcommand{\ev}[0]{\operatorname{ev}}
\newcommand{\fin}[0]{{\subset\!\!\!\subset}}
\newcommand{\diam}[0]{\operatorname{diam}}
\newcommand{\Hom}[0]{\operatorname{Hom}}
\newcommand{\Aut}[0]{\operatorname{Aut}}
\newcommand{\del}[0]{\partial}
\newcommand{\dimnuc}[0]{\dim_{\mathrm{nuc}}}
\newcommand{\dr}[0]{\operatorname{dr}}
\newcommand{\dimrok}[0]{\dim_{\mathrm{Rok}}}
\newcommand{\dimrokcyc}[0]{\dim_{\mathrm{Rok}}^{\mathrm{cyc}}}
\newcommand{\dimrokcycc}[0]{\dim_{\mathrm{Rok}}^{\mathrm{cyc,c}}}
\newcommand{\dimnuceins}[0]{\dimnuc^{\!+1}}
\newcommand{\dreins}[0]{\dr^{\!+1}}
\newcommand{\dimrokeins}[0]{\dimrok^{\!+1}}
\newcommand{\reldimrok}[2]{\dimrok(#1~|~#2)}
\newcommand{\mdim}[0]{\operatorname{mdim}}
\newcommand*\onto{\ensuremath{\joinrel\relbar\joinrel\twoheadrightarrow}} 
\newcommand*\into{\ensuremath{\lhook\joinrel\relbar\joinrel\rightarrow}}  
\newcommand{\im}[0]{\operatorname{Im}}
\newcommand{\dst}[0]{\displaystyle}
\newcommand{\cstar}[0]{$\mathrm{C}^*$}
\newcommand{\dist}[0]{\operatorname{dist}}
\newcommand{\ue}[0]{{\approx_{\mathrm{u}}}}
\newcommand{\mue}[0]{{\approx_{\mathrm{mu}}}}
\newcommand{\End}[0]{\operatorname{End}}
\newcommand{\Ell}[0]{\operatorname{Ell}}
\newcommand{\gr}[0]{\operatorname{gr}}
\newcommand{\Ost}[0]{\mathcal{O}_\infty^{\mathrm{st}}}
\newcommand{\Bst}[0]{\mathcal{B}^{\mathrm{st}}}
\newcommand{\inv}[0]{\operatorname{Inv}}
\newcommand{\ann}[0]{\operatorname{Ann}}

\newtheorem{satz}{Satz}[section]		
\newtheorem{cor}[satz]{Corollary}
\newtheorem{lemma}[satz]{Lemma}
\newtheorem{prop}[satz]{Proposition}
\newtheorem{theorem}[satz]{Theorem}
\newtheorem*{theoreme}{Theorem}

\theoremstyle{definition}
\newtheorem{defi}[satz]{Definition}
\newtheorem*{defie}{Definition}
\newtheorem{defprop}[satz]{Definition \& Proposition}
\newtheorem{nota}[satz]{Notation}
\newtheorem*{notae}{Notation}
\newtheorem{rem}[satz]{Remark}
\newtheorem*{reme}{Remark}
\newtheorem{example}[satz]{Example}
\newtheorem{defnot}[satz]{Definition \& Notation}
\newtheorem{question}[satz]{Question}
\newtheorem*{questione}{Question}

\newenvironment{bew}{\begin{proof}[Proof]}{\end{proof}}

\begin{abstract} 
This paper serves as a source of examples of Rokhlin actions or locally representable actions of finite groups on \cstar-algebras satisfying a certain UHF-absorption condition.
We show that given any finite group $G$ and a separable, unital \cstar-algebra $A$ that absorbs $M_{|G|^\infty}$ tensorially, one can lift any group homomorphism $G\to\Aut(A)/_{\ue}$ to an honest Rokhlin action $\gamma$ of $G$ on $A$. Unitality may be dropped in favour of stable rank one or being stable. If $A$ belongs to a certain class of \cstar-algebras that is classifiable by a suitable invariant (e.g. $K$-theory), then in fact every $G$-action on the invariant lifts to a Rokhlin action of $G$ on $A$.
For the crossed product \cstar-algebra $A\rtimes_\gamma G$ of a Rokhlin action on a UHF-absorbing \cstar-algebra, an inductive limit decomposition is obtained in terms of $A$ and $\gamma$.
If $G$ is assumed to be abelian, then the dual action $\hat{\gamma}$ is locally representable in a very strong sense.  We then show how some well-known constructions of finite group actions with certain predescribed properties can be recovered and extended by the main results of this paper, when paired with known classification theorems. Among these is Blackadar's famous construction of symmetries on the CAR algebra whose fixed point algebras have non-trivial $K_1$-groups. Lastly, we use the results of this paper to reduce the UCT problem for separable, nuclear \cstar-algebras to a question about certain finite group actions on $\CO_2$.
\end{abstract}

\maketitle


\setcounter{section}{-1}

\section{Introduction}
\noindent
The study of noncommutative dynamical systems is a very active area of research with lots of open problems and things yet to discover. Be it either the investigation of crossed product \cstar-algebras or the actions themselves, \cstar-dynamical systems spark great interest from the viewpoint of classification theory. In the particular case of finite group actions, Phillips has been steadily paving the way towards a classification theory of pointwise outer actions on unital Kirchberg algebras, building on ideas of ordinary classification theory of Kirchberg algebras \cite{PhiC, KiC} and making use of equivariant absorption theorems in the spirit of \cite{KiPhi}. The key ideas of the latter have already been demonstrated in \cite{GoldIz}. Concerning the case of poly-$\IZ$ groups, initial positive results in this direction are due to Nakamura \cite{Nakamura}, Izumi-Matui \cite{IzumiMatuiZ2} and is currently developed further by Izumi-Matui, see \cite{IzumiMatuiOWR}.

An important question related to such a classification theory is how large the range of the objects is that one wishes to classify. For example, if a certain class of \cstar-algebras is classified by $K$-theoretic data and one wishes to study finite group actions on such, this begs the question of whether every group action on the $K$-theory of a \cstar-algebra in this class can be lifted to an honest group action on the \cstar-algebra. To be more precise, one poses the following question:

\begin{questione} If $A$ belongs to a class of \cstar-algebras that is classifiable by a functor $\inv$ (in a suitable sense) and $\sigma: G\curvearrowright \inv(A)$ is an action of a finite group on the invariant of $A$, does there exist an action $\alpha: G\curvearrowright A$ with $\inv(\alpha)=\sigma$? 
\end{questione}

It seems that, compared to the recent progress in the Elliott classification program, satisfactory answers to this question are very scarce. We note that even within the setting of $G=\IZ_n$ and $A$ being an AF algebra, this question appears to be still open, see \cite[10.11.3]{BlaKK}.
 
To the authors' best knowledge, only the case of actions on unital UCT Kirchberg algebras has been successfully examined so far, in which the invariant boils down to $K$-theory. Symmetries, i.e.\ $\IZ_2$-actions, were first considered in \cite{BKP}. The main result asserted that if the \cstar-algebras in question are unital UCT Kirchberg algebras in Cuntz standard form, then the above question has an affirmative answer. This was extended to finite cyclic group actions of prime order in \cite{Sp}, where also the assumption of the Cuntz standard form could be dropped. Finally in \cite{Kat}, this was further extended to actions of finite groups whose Sylow subgroups are cyclic. 

Certain difficulties in the aforementioned papers arose from the fact that a slightly stronger question was asked than just the above. For example, a related question is if there exists a lift $\alpha$ like in the above question that comes from a (partial) split of the classifying functor applied to $\sigma$. Another stronger question, which was answered in \cite{BKP} for symmetries, is how large the range of actions is with respect to equivariant $K$-theory instead of ordinary $K$-theory.

In this paper, we take another viewpoint that is also suitable for actions on not necessarily classifiable \cstar-algebras: For a \cstar-algebra $A$ and a finite group $G$, when can a homomorphism $G\to\Aut(A)/_\ue$ lift to an honest action of $G$ on $A$? While we must certainly impose certain restrictions on $A$, it turns out that a sufficient criterion is common enough to produce a variety of interesting examples. Incidentally, the lifts we construct all have the Rokhlin property. 

\begin{theoreme}[\ref{range}] Let $G$ be a finite group and $A$ a separable \cstar-algebra that absorbs the UHF algebra $M_{|G|^\infty}$ tensorially. Assume that $A$ is either unital, stable or has stable rank one. Then any homomorphism $G\to\Aut(A)/_\ue$ lifts to a Rokhlin action of $G$ on $A$.
\end{theoreme}
 
Notably, no simplicity or classifiability assumption on $A$ is needed in order to prove this. Looking back to the above question about the range of the invariant of $G$-actions on classifiable \cstar-algebras, this more general viewpoint is (a priori) weaker because the Elliott invariant alone is usually not strong enough to distinguish between approximate unitary equivalence classes of $*$-homo\-morphisms. However, a strong enough classification result paired with UCT usually implies that the canonical map $\Aut(A)/_\ue\to \Aut(\inv(A))$ is not only surjective, but has a split. For example, this works for Kirchberg algebras or simple, nuclear TAF algebras, see \cite{Lin0, LinC}. This enables one to reduce the above question about the range of the invariant of actions to the question of being able to lift homomorphisms $G\to\Aut(A)/_\ue$ to honest group actions on $A$. In particular, \ref{range} allows us to prove the following result:

\begin{theoreme}[\ref{Rok range Kirchberg} and \ref{Rok range TAF}] 
Let $G$ be a finite group and $A$ a separable, unital, nuclear and simple \cstar-algebra with $A\cong M_{|G|^\infty}\otimes A$. Assume that $A$ satisfies the UCT and is either purely infinite or TAF. Then Rokhlin actions of $G$ on $A$ exhaust all $G$-actions on the (ordered) $K$-theory of $A$.
\end{theoreme}

The Rokhlin property for finite groups was pioneered by Herman and Jones in the setting of UHF algebras, see \cite{HermanJones1, HermanJones2}. They followed the argument of Rokhlin actions developed by Connes \cite{Connes1, Connes2} and Ocneanu \cite{Ocneanu} for amenable groups in the von Neumann algebra setting. In his remarkable papers \cite{Izumi, Izumi2}, Izumi then introduced the Rokhlin property for finite group actions on unital \cstar-algebras. Finite group actions with the Rokhlin property are very rigid. For example, two Rokhlin actions on a unital, separable \cstar-algebra, which are pointwise approximately unitarily equivalent, are conjugate by an approximately inner automorphism, see \cite{Izumi}. In fact, the same result holds for non-unital \cstar-algebras as well; first proved in \cite{Nawata} for \cstar-algebras of almost stable rank one, and then proved in general in \cite{GardSant}. In particular, this means that up to conjugacy, we must construct every possible Rokhlin action of $G$ on $A$ in order to prove \ref{range}. This enables us to make some non-trivial observations concerning the structure of the crossed product \cstar-algebras $A\rtimes_\gamma G$, when $\gamma$ is an arbitrary action with the Rokhlin property and $A$ absorbs $M_{|G|^\infty}$ tensorially.
We will also examine the dual actions of Rokhlin actions more closely, in the case that the acting group is abelian. It turns out that when $G$ and $A$ are as in \ref{range} with $G$ abelian and $\gamma: G\curvearrowright A$ has the Rokhlin property, then $\hat{\gamma}$ is locally representable in a very strong sense. In particular, it is an inductive limit type action, where the building blocks and the actions on the building blocks are fairly easy to grasp. 

The paper is organized as follows.
In the first section, we present some notation and definitions that we will use throughout the paper and remind the reader of some standard techniques.

In the second section, we prove that under the aforementioned conditions on a finite group $G$ and a \cstar-algebra $A$, every homomorphism $G\to\Aut(A)/_\ue$ lifts to a Rokhlin action of $G$ on $A$. We do this by constructing an inductive limit model system on which we define the action. Combining UHF-absorption of $A$ with an Elliott intertwining argument, we prove that the model system just recovers the given \cstar-algebra $A$. Several consequences of this model system are deduced, for example an inductive limit decomposition of the crossed product \cstar-algebra $A\rtimes_\gamma G$, when $A$ is $M_{|G|^\infty}$-absorbing and $\gamma$ is an arbitrary Rokhlin action of $G$ on $A$.

In the third section, we examine the dual actions of the Rokhlin actions treated in Section 2, when the acting group is abelian. As is known from \cite{Izumi, Nawata}, such a dual action is always approximately representable. Under the UHF-absorption condition, however, it turns out that such a dual action is even locally representable in a very strong sense.

In the fourth section, we treat some interesting examples that arise as consequences of the main results of this paper. It turns out that Blackadar's famous construction of \cite{Blackadar} can be recovered and extended to a more general setting by combining \ref{Rok range TAF} with Lin's classification theory of TAF algebras \cite{LinC}. Moreover, an analogous result as the range result for certain symmetries on the CAR algebra \cite[6.2.4]{Blackadar} follows for actions of all finite abelian groups, particularly for order $p$ automorphisms for all $p>2$. In his remarkable work on the Rokhlin property of finite group actions on unital \cstar-algebras \cite{Izumi, Izumi2}, Izumi gave a range result of approximately representable actions of finite cyclic groups with prime power order on $\CO_2$, see \cite[4.8(3), 4.9]{Izumi} and \cite[6.4]{Izumi2}. This extends quite naturally to finite cyclic groups of any order, and in fact all finite abelian groups, by combining \ref{Rok range Kirchberg} with Kirchberg-Phillips classification. On the front of non-unital \cstar-algebras, one can also combine \ref{range} with Robert's classification theorem \cite{Robert} to construct actions on stably projectionless \cstar-algebras. For example, the construction carried out in \cite[5.6]{Nawata}, which was intended as a stably projectionless analogue of \cite[4.7]{Izumi}, can be recovered. At last, we sketch how to reduce the UCT problem for separable, nuclear \cstar-algebras to the question of whether one can leave the UCT class by passing to crossed product \cstar-algebras of $\CO_2$ by finite group actions. We will even reduce it to the more special setting where one only needs to consider certain locally representable actions of $\IZ_2$ and $\IZ_3$ on $\CO_2$. It is possible that this is already known to some experts (see \cite[23.15.12(d)]{BlaKK}), but to the authors' knowledge, there is no proof in the literature. 

The authors would like to thank Dominic Enders for a number of very useful remarks on an earlier preprint version of this paper. Moreover, the authors are grateful to the referee for a lot of suggestions that have led to a considerable overall improvement of this paper.




\section{Preliminaries}

\begin{samepage}
\nota Unless specified otherwise, we will stick to the following notations throughout the paper.
\begin{itemize}
\item $G$ denotes a fixed finite group. 
\item For a natural number $p\geq 2$, we denote $\IZ_p = \IZ/p\IZ$.
\item $A$ and $B$ are separable \cstar-algebras.
\item $\tilde{B}$ denotes the unitalization of $B$, if $B$ is non-unital. If $B$ already is unital, then $\tilde{B}$ is $B$. $\CM(B)$ denotes the multiplier algebra of $B$.
\item If $M$ is some set and $F\subset M$ is a finite subset, we write $F\fin M$.
\item For $\eps>0$ and $a,b$ in some normed space, we write $a=_\eps b$ for $\|a-b\|\leq\eps$.
\end{itemize}
\end{samepage}

First, we recall some needed definitions. Note that, in the next definition, we are in principle following \cite{Ki}, but with a slightly enhanced notation to point out that we work with ordinary sequence algebras instead of ultrapowers.

\defi \label{central sequence def} Let $A$ be a separable \cstar-algebra. Let $B\subset A_\infty$ be a sub-\cstar-algebra. Define
\[
\ann(B,A_\infty) = \set{x\in A_\infty ~|~ xb=bx=0~\text{for all}~b\in B.}
\]
and
\[
F_\infty(B,A) = (A_\infty \cap B')/\ann(B, A_\infty) 
\]
as the central sequence algebra of $A$ relative to $B$. One writes $F_\infty(A) = F_\infty(A,A)$.
If $\phi\in\Aut(A)$, we denote the induced automorphism on $F_\infty(A)$, given by componentwise application of $\phi$, by $\phi_\infty$. This notation carries over to (finite) group actions.

\begin{defi}[following {\cite[3.1]{Nawata}}] \label{rokprop nawata}
Let $G$ be a finite group and $A$ a separable \cstar-algebra. Let $\alpha: G\curvearrowright A$ be an action. Then $\alpha$ is said to have the Rokhlin property, if there exists an equivariant and unital $*$-homomorphism
\[
(\CC(G),G\text{-shift})\longrightarrow (F_\infty(A), \alpha_\infty).
\]
In this context, we shall also call $\alpha$ a Rokhlin action.
\end{defi}

\begin{rem} \label{rokprop} 
Denote by $\set{e_g}_{g\in G}$ the characteristic functions of the points $g\in G$ inside $\CC(G)$. By representing the images of each $e_g$ in $F_\infty(A)$ by some bounded sequences in $\ell^\infty(\IN, A)$ it is easy to check the following:

Let $G$ be a finite group, $A$ a separable \cstar-algebra and $\alpha: G\curvearrowright A$ an action via automorphisms.
Then $\alpha$ has the Rokhlin property, if and only if the following holds:
For all $\eps>0$ and $F\fin A$, there exist positive contractions 
$\set{a_g}_{g\in G}$ in $A$ satisfying\vspace{-1mm}
\begin{itemize}
\item[(1)] $\dst \Bigl( \sum_{g\in G} a_g\Bigl)x=_\eps x$ for all $x\in F$.
\item[(2)] $\alpha_g(a_h)x =_\eps a_{gh} x$\quad for all $g,h\in G$ and $x\in F$.\vspace{1mm}
\item[(3)] $\|a_g a_h x\|\leq\eps$\quad for all $g\neq h$ in $G$ and $x\in F$.\vspace{1mm}
\item[(4)] $\|[a_g,x]\|\leq\eps$\quad for all $g\in G$ and $x\in F$.\vspace{1mm}
\end{itemize}
\end{rem}

We refer to \cite{Santiago} for a more detailed survey and comparison between several known variants of the Rokhlin property. In particular, it follows from \cite{Santiago} that the Rokhlin elements can actually be chosen in such a way that the relations (2) and (3) hold approximately in norm instead of the strict topology. However, as the above formulation is a priori the weakest version, it is also the easiest to verify.

\defi \begin{itemize} \label{approx unitary equ}
\item Two $*$-homomorphisms $\phi,\psi: A\to B$ are called approximately multiplier unitarily equivalent, written $\phi\mue\psi$, if there are unitaries $u_n\in\CU(\CM(B))$ such that
\[\psi(a) = \lim_{n\to\infty} u_n\phi(a)u_n^*\quad\text{for all}~a\in A. \]
\item Two $*$-homomorphisms $\phi,\psi: A\to B$ are called approximately unitarily equivalent, written $\phi\ue\psi$, if there are unitaries $u_n\in\CU(\tilde{B})$ such that
\[\psi(a) = \lim_{n\to\infty} u_n\phi(a)u_n^*\quad\text{for all}~a\in A. \]
\item A $*$-endomorphism on $A$ is called approximately inner, if it is approximately unitarily equivalent to $\id_A$. A $*$-endomorphism on $A$ is called approximately multiplier inner, if it is approximately multiplier unitarily equivalent to $\id_A$.
\end{itemize}

\rem Obviously, the relations $\phi\mue\psi$ and $\phi\ue\psi$ are the same, if $B$ is unital. In general, the two definitions are different.
However, it is known that $\phi\mue\psi$ implies $\phi\ue\psi$ in the cases that $B$ is either stable or has stable rank one.

It is well-known that $\Aut(A)/_{\ue}$ defines a group with the operation $[\theta_1]\cdot [\theta_2]=[\theta_1\circ\theta_2]$. The same is true for $\Aut(A)/_{\mue}$.

\rem[see \cite{TomsWinter}] \label{uhfabs}
\begin{enumerate}
\item Let $q\geq 2$ be any number and $A\cong M_{q^\infty}\otimes A$. Then the canonical embedding
\[\eins_{M_{q^\infty}}\otimes\id_A: A\to M_{q^\infty}\otimes A,\quad a\mapsto\eins_{M_{q^\infty}}\otimes a \]
is approximately multiplier unitarily equivalent to an isomorphism.
\item Let $\phi$ be such an isomorphism. Then one has that
\[ \eins_{M_{q^n}}\otimes\Bigl(\phi^{-1}\circ (\eins_{M_{q^\infty}}\otimes \id_{M_{q^n}\otimes A})\Bigl): M_{q^n}\otimes A\to M_{q^n}\otimes A \]
is approximately multiplier inner for all $n\in\IN$.
\end{enumerate}

\begin{lemma}[see {\cite[2.3.4 and 2.3.3]{Rordam}} and {\cite{RordamKorrektur}}] \label{intertwining}
\begin{itemize}
\item[(1)] Let $A$ and $B$ be separable \cstar-algebras. Let $\phi_0: A\to B$ and $\psi_0: B\to A$ be two $*$-homomorphisms such that $\psi_0\circ\phi_0\ue\id_A$ and $\phi_0\circ\psi_0\ue \id_B$. Then there is an isomorphism $\phi: A\to B$ with $\phi\ue\phi_0$ and $\phi^{-1}\ue\psi_0$. The same statement holds for $\mue$ instead of $\ue$.
\item[(2)] Let $\set{A_n}_{n\in\IN}$ be a sequence of separable \cstar-algebras. For every $n$, let
\[\Phi_n: A_n\to A_{n+1}\quad\text{and}\quad \Psi_n: A_n\to A_{n+1} \]
be two $*$-homomorphisms. If $\Phi_n\ue\Psi_n$ for all $n$, then there exists an isomorphism
\[ \Theta: \lim_{\longrightarrow} \set{ A_n, \Psi_n} \longrightarrow \lim_{\longrightarrow} \set{A_n,\Phi_n} \]
such that
\[\Phi_{n+1,\infty}\circ\Psi_n\ue\Theta\circ\Psi_{n,\infty}\quad\text{and}\quad \Psi_{n,\infty}\ue\Theta^{-1}\circ\Phi_{n,\infty} \]
for all $n\in\IN$. The same statement holds for $\mue$ instead of $\ue$.
\end{itemize}
\end{lemma}


\section{An existence theorem for Rokhlin actions on UHF absorbing \cstar-algebras}
\noindent
In this section, we prove that for every finite group $G$ and certain UHF absorbing \cstar-algebras $A$, every homomorphism $G\to\Aut(A)/_\ue$ lifts to a Rokhlin action of $G$ on $A$. 

\nota
\begin{itemize} 
\item For a \cstar-algebra $A$ and a finite index set $I$, denote
\[A^{\oplus I} = \bigoplus_{i\in I} A. \]
\item For a finite group $G$, let $\set{e_{g,h}}_{g,h\in G}$ denote the generating matrix units of $M_{|G|}$.
\end{itemize}

\begin{lemma} \label{psilim} 
Let $G$ be a finite group and $A$ a separable \cstar-algebra with $A\cong M_{|G|^\infty}\otimes A$. Assume that either $A$ is unital, stable, or has stable rank one. Let $\set{\beta_g}_{g\in G}\subset\Aut(G)$ be a collection of automorphisms. Let $A^{(n)}=M_{|G|^{n-1}}\otimes A$ and $\beta^{(n)}_g=\id_{M_{|G|^{n-1}}}\otimes\beta_g$ for all $n\in\IN, g\in G$ and consider the inductive system
\[
\Psi_n: A^{(n) \oplus G} \to A^{(n+1)\oplus G},\quad \Psi_n((x_g)_{g\in G}) = \Bigl( \sum_{h\in G} e_{h,h}\otimes (\beta_g^{(n)})^{-1}\circ\beta_h^{(n)}(x_h) \Bigl)_{g\in G}.
\]
Denote $\dst A_\Psi = \lim_{\longrightarrow} \set{ A^{(n)\oplus G}, \Psi_n }$.
Then the embedding 
\[ \iota_\infty: A\into A_\Psi,\quad x\longmapsto \Psi_{1,\infty}\Bigl(((\beta_g)^{-1}(x))_{g\in G}\Bigl) \]
is approximately unitarily equivalent to an isomorphism.
\end{lemma}
\begin{proof}
For every $n$, define a $*$-homomorphism
\[P_n: A^{(n)\oplus G}\to M_{|G|^\infty}\otimes A\quad\text{via}\quad
P_n((x_g)_{g\in G}) = \eins_{M_{|G|^\infty}}\otimes  \Bigl(\sum_{g\in G} e_{g,g}\otimes \beta^{(n)}_g(x_g)\Bigl).\]
Observe for all $n$ and $(x_g)_g\in A^{(n)\oplus G}$ that
\[ \begin{array}{ccl}
\multicolumn{3}{l}{ P_{n+1}\circ\Psi_n((x_g)_{g\in G}) }\\
\hspace{10mm} &=& \dst P_{n+1}\left( \Bigl( \sum_{h\in G} e_{h,h}\otimes (\beta_g^{(n)})^{-1}\circ\beta_h^{(n)}(x_h) \Bigl)_{g\in G}\right) \\
&=& \dst \eins_{M_{|G|^\infty}}\otimes\left( \sum_{g\in G} e_{g,g}\otimes \beta^{(n+1)}_g \Bigl( \sum_{h\in G} e_{h,h}\otimes (\beta_g^{(n)})^{-1}\circ\beta_h^{(n)}(x_h) \Bigl)  \right) \\
&=& \dst \eins_{M_{|G|^\infty}}\otimes\left( \sum_{g,h\in G} e_{g,g}\otimes e_{h,h}\otimes \beta_h^{(n)}(x_h) \right) \\
&=& P_n((x_g)_{g\in G})
.\end{array}\]
Hence, the $P_n$ give rise to a well-defined $*$-homomorphism $P: A_\Psi\to M_{|G|^\infty}\otimes A$. By \ref{uhfabs}, we pick an isomorphism $\phi: A\to M_{|G|^\infty}\otimes A$ with $\phi\ue\eins_{M_{|G|^\infty}}\otimes\id_A$. We obtain a $*$-homomorphism $p = \phi^{-1}\circ P: A_\Psi\to A$.

Now define $\iota: A\into A^{(1)\oplus G}$ via $\iota(x)=((\beta_g)^{-1}(x))_{g\in G}$. For $g\in G$, let $q_g^{(n)}: A^{(n)\oplus G}\to A^{(n)}$ denote the canonical projection onto the $g$-component. Check that
\[ \begin{array}{cll}
p\circ\iota_\infty &=& p\circ\Psi_{1,\infty}\circ\iota \vspace{2mm}\\
&=& \phi^{-1}\circ P_1\circ\iota \\
&=& \dst\phi^{-1}\circ\left( \eins_{M_{|G|^\infty}}\otimes\Bigl( \sum_{g\in G} e_{g,g}\otimes\beta_g\circ q^{(1)}_g\circ\iota \Bigl) \right) \\
&=& \dst\phi^{-1}\circ\left( \eins_{M_{|G|^\infty}}\otimes\Bigl( \sum_{g\in G} e_{g,g}\otimes\id_A \Bigl) \right) \\
&=& \phi^{-1}\circ(\eins_{M_{|G|^\infty}}\otimes\id_A) \\
&\ue& \id_A
.\end{array}\]
Next, we wish to show that also $\iota_\infty\circ p\ue\id_{A_\Psi}$. Observe that for all $n$, we have the identity $\Psi_{1,n}\circ\iota = \eins_{M_{|G|^{n-1}}}\otimes\iota$.
Using this, we calculate for all $n$:
\[ \begin{array}{cll}
\multicolumn{3}{l}{ \Psi_{1,n+1}\circ\iota\circ\phi^{-1}\circ P_n} \\
&=& (\eins_{M_{|G|^{n}}}\otimes\iota)\circ\phi^{-1}\circ P_n \\
&=& \dst \bigoplus_{g\in G} (\beta_g^{(n+1)})^{-1}\circ (\eins_{M_{|G|^{n}}}\otimes\id_A)\circ\phi^{-1}\circ P_n \\
&=& \dst \bigoplus_{g\in G} (\beta_g^{(n+1)})^{-1}\circ (\eins_{M_{|G|^{n}}}\otimes\id_A)\circ\phi^{-1}\circ \left( \eins_{M_{|G|^\infty}}\otimes  \Bigl(\sum_{h\in G} e_{h,h}\otimes \beta^{(n)}_h\circ q^{(n)}_h\Bigl) \right) \\
&\hspace{-1.5mm}\stackrel{\ref{uhfabs}(2)}{\ue}& \dst \bigoplus_{g\in G} (\beta_g^{(n+1)})^{-1}\circ \Bigl(\sum_{h\in G} e_{h,h}\otimes \beta^{(n)}_h\circ q^{(n)}_h\Bigl) \\
&=& \dst \bigoplus_{g\in G} \Bigl( \sum_{h\in G} e_{h,h} \otimes (\beta_g^{(n)})^{-1}\circ\beta_h^{(n)}\circ q^{(n)}_h \Bigl) \\
&=& \Psi_n
.\end{array}\]
Hence, it is clear that $\iota_\infty\circ p = \Psi_{1,\infty}\circ\iota\circ\phi^{-1}\circ P\ue\id_{A_\Psi}$. The proof is complete with an application of \ref{intertwining}(1).
\end{proof} 

Using the previous lemma, we can prove an existence result for Rokhlin actions on \cstar-algebras as in \ref{psilim}.

{
\theorem \label{range} Let $G$ be a finite group and $A$ a separable \cstar-algebra with $A\cong M_{|G|^\infty}\otimes A$. Assume that either $A$ is unital, stable, or has stable rank one. Let $\set{\beta_g}_{g\in G}\subset\Aut(A)$ be a collection of automorphisms that defines a $G$-action on $A$ up to approximate unitary equivalence, i.e. 
$\beta_g\circ\beta_h\ue\beta_{gh}$ for all $g,h\in G$.
Then there exists a Rokhlin action $\gamma: G\curvearrowright A$ such that $\gamma_g\ue\beta_g$ for all $g\in G$.
}
\begin{proof}
We make use of \ref{psilim} and its proof. Adopt the notation for the $A^{(n)}$, the connecting maps $\Psi_n$ and the limit $A_\Psi$. Also recall the definition of $\iota: A\to A^{(1)\oplus G}$ and $\iota_\infty: A\to A_\Psi$.

To define the action $\gamma$, we will use a model system for $A$ similar to $A_\Psi$. Then we will compare the new model system with $A_\Psi$ and we will see that the action, when pulled back to $A$, will be an action with the desired properties.

For all $n$, we define
\[ 
\Phi_n: A^{(n)\oplus G}\to A^{(n+1)\oplus G}\quad\text{via}\quad \Phi_n((x_g)_{g\in G})=\Bigl(\sum_{h\in G} e_{h,h}\otimes\beta_h^{(n)}(x_{gh}) \Bigl)_{g\in G}. 
\]
Consider the limit $\dst A_\Phi=\lim_{\longrightarrow} \set{ A^{(n)\oplus G}, \Phi_n}$. For all $n\in\IN$, we define a $G$-action $\gamma'^{(n)}: G\curvearrowright A^{(n)\oplus G}$ by the shift $\gamma'^{(n)}_f((x_g)_{g\in G}) = (x_{f^{-1}g})_{g\in G}$. One checks for all $n$ and $f\in G$ that
\[ 
\begin{array}{cll}
\gamma'^{(n+1)}_f\circ\Phi_n((x_g)_{g\in G}) &=& \dst
\gamma'^{(n+1)}_f\left(\Bigl(\sum_{h\in G} e_{h,h}\otimes\beta_h^{(n)}(x_{gh}) \Bigl)_{g\in G} \right) \\
&=& \dst \Bigl(\sum_{h\in G} e_{h,h}\otimes\beta_h^{(n)}(x_{f^{-1}gh}) \Bigl)_{g\in G} \vspace{2mm}\\
&=& \Phi_n( (x_{f^{-1}g})_{g\in G} ) \vspace{2mm}\\
&=& \Phi_n\circ\gamma'^{(n)}_f((x_g)_{g\in G})
.\end{array}
\]
Therefore, the actions $\gamma'^{(n)}$ extend to an action $\gamma': G\curvearrowright A_\Phi$ on the inductive limit. It is obvious that $\gamma'$ has the Rokhlin property, since the building block actions $\gamma'^{(n)}$ are of the form $\id_{A^{(n)}}\otimes (G\text{-shift}):G\curvearrowright A^{(n)}\otimes \CC(G)$.

Now let $\lambda: G\to\CU(M_{|G|})$ be the left-regular representation of $G$ given by $\lambda(g)=\sum_{h\in G} e_{gh,h}$. Note that $\lambda(g)e_{h,k}\lambda(g)^*=e_{gh,gk}$ for all $g,h,k\in G$. Then the properties of the $\beta_g$ show that 
\[
\begin{array}{ccl}
\Psi_n &\ue&  \dst\ad\Bigl( \bigoplus_{g\in G} \lambda(g^{-1})\otimes \eins_{A^{(n)}}\Bigl)\circ \Psi_n  \\\\
&=& \dst \ad\Bigl( \bigoplus_{g\in G} \lambda(g^{-1})\otimes \eins_{A^{(n)}}\Bigl)\circ \Bigl(\bigoplus_{g\in G} \sum_{h\in G} e_{h,h}\otimes(\beta_g^{(n)})^{-1}\circ\beta_h^{(n)}\circ q_h^{(n)} \Bigl) \\\\
&=& \dst\bigoplus_{g\in G} \sum_{h\in G} e_{g^{-1}h,g^{-1}h}\otimes(\beta_g^{(n)})^{-1}\circ\beta_h^{(n)}\circ q_h^{(n)} \\\\
&\ue& \dst\bigoplus_{g\in G} \sum_{h\in G} e_{g^{-1}h,g^{-1}h}\otimes\beta_{g^{-1}h}^{(n)}\circ q_h^{(n)} \\\\
&=& \dst\bigoplus_{g\in G} \sum_{h\in G} e_{h,h}\otimes\beta_{h}^{(n)}\circ q_{gh}^{(n)} \quad = \Phi_n
\end{array}
\]
for all $n$. So we may apply \ref{intertwining}(2) to find an isomorphism
$ \Theta: A_\Psi \to A_\Phi$
with
\[
\Phi_{n+1,\infty}\circ\Psi_n\ue\Theta\circ\Psi_{n,\infty}\quad\text{and}\quad \Psi_{n,\infty}\ue\Theta^{-1}\circ\Phi_{n,\infty} 
\]
for all $n\in\IN$.
Moreover, we apply \ref{psilim} to find an isomorphism $\sigma: A\to A_\Psi$ with $\sigma\ue\iota_\infty$. We define
\[ 
\gamma: G\curvearrowright A\quad\text{by}\quad \gamma_g = \sigma^{-1}\circ\Theta^{-1}\circ\gamma'_g\circ\Theta\circ\sigma\quad\text{for all}~g\in G.
\]
Since $\gamma'$ had the Rokhlin property, so has $\gamma$. For all $f\in G$, we now observe that
\[ 
\begin{array}{cll}
\gamma_f &\ue& \sigma^{-1}\circ\Theta^{-1}\circ\gamma'_f\circ\Theta\circ\Psi_{1,\infty}\circ\iota \vspace{2mm}\\
&\ue&\sigma^{-1}\circ\Theta^{-1}\circ\gamma'_f\circ\Phi_{2,\infty}\circ\Psi_1\circ\iota \vspace{2mm}\\
&=& \sigma^{-1}\circ\Theta^{-1}\circ\Phi_{2,\infty}\circ\gamma'^{(2)}_f\circ \Psi_1\circ\iota \vspace{2mm}\\
&=& \sigma^{-1}\circ\Theta^{-1}\circ\Phi_{2,\infty}\circ\gamma'^{(2)}_f\circ (\eins_{M_{|G|}}\otimes\iota) \vspace{2mm}\\
&=& \dst\sigma^{-1}\circ\Theta^{-1}\circ\Phi_{2,\infty}\circ\gamma'^{(2)}_f\circ\Bigl(\bigoplus_{h\in G} \eins_{M_{|G|}}\otimes(\beta_h)^{-1}\Bigl) \vspace{2mm}\\
&=& \dst\sigma^{-1}\circ\Theta^{-1}\circ\Phi_{2,\infty}\circ
\Bigl(\bigoplus_{h\in G} \eins_{M_{|G|}}\otimes(\beta_{f^{-1}h})^{-1} \Bigl) \vspace{2mm}\\
&\ue& \dst\sigma^{-1}\circ\Theta^{-1}\circ\Phi_{2,\infty}\circ
\Bigl(\bigoplus_{h\in G} \eins_{M_{|G|}}\otimes(\beta_h)^{-1}\circ \beta_f \Bigl) \vspace{2mm}\\
&=& \dst\sigma^{-1}\circ\Theta^{-1}\circ\Phi_{2,\infty}\circ(\eins_{M_{|G|}}\otimes\iota)\circ\beta_f \vspace{2mm}\\
&\ue& \sigma^{-1}\circ\Psi_{2,\infty}\circ\Psi_1\circ\iota\circ\beta_f \vspace{2mm}\\
&=& \sigma^{-1}\circ\iota_\infty\circ\beta_f \quad\ue\quad \beta_f
.\end{array}
\]
This finishes the proof.
\end{proof}

Recall the following particularly useful rigidity result for Rokhlin actions:

{
\theorem[see {\cite[3.5]{Izumi}, \cite[3.5]{Nawata}} and {\cite[Thm 3.4]{GardSant}} ] \label{rok ue}
Let $A$ be a separable \cstar-algebra. Let $\alpha^{(0)},\alpha^{(1)}: G\curvearrowright A$ be two Rokhlin actions of a finite group with $\alpha_g^{(0)}\ue\alpha_g^{(1)}$ for all $g\in G$. Then there is an approximately inner automorphism $\phi\in\Aut(A)$ such that $\alpha_g^{(1)}\circ\phi=\phi\circ\alpha_g^{(0)}$ for all $g\in G$. 
}

\defi Let $H_1, H_2$ be two discrete groups. Recall that two group homomorphisms $\psi_1,\psi_2: H_1\to H_2$ are called conjugate, if there exists $g\in H_2$ with $\psi_2(h)=g\psi_1(h)g^{-1}$ for all $h\in H_1$. We denote $\Hom(H_1,H_2)$ modulo the conjugacy relation by $\quer{\Hom}(H_1,H_2)$.

\defi Let $G$ be a finite group and $A$ a \cstar-algebra. Let $\CR_G(A)$ denote the set of all Rokhlin actions of $G$ on $A$ and let $\quer{\CR}_G(A)$ be the set of all conjugacy classes of Rokhlin actions of $G$ on $A$.

{
\cor \label{ue range} Let $G$ be a finite group and $A$ a separable \cstar-algebra with $A\cong M_{|G|^\infty}\otimes A$. Assume that $A$ is either unital, stable, or has stable rank one. Then the natural map
\[
\CR_G(A)\longrightarrow\Hom(G, \Aut(A)/_\ue),\quad [g\mapsto \alpha_g]\longmapsto [g\mapsto [\alpha_g]_{\ue}] 
\]
is surjective and induces a bijection
\[
\quer{\CR}_G(A)\longrightarrow\quer{\Hom}(G, \Aut(A)/_\ue). 
\]
}
\begin{proof}
Surjectivity follows directly from \ref{range} and injectivity follows from \ref{rok ue}.
\end{proof}

The following fact appears to be well-known among experts, and a special case has been proved in \cite{GardSant}. However, to the best of the authors' knowledge, there is no proof of the following fact in the current literature in this generality, so we shall give one here for the reader's convenience. Note that this rounds out the above result \ref{ue range} quite nicely.

{
\prop Let $G$ be a finite group and $A$ a separable \cstar-algebra. If there exists a Rokhlin action $\alpha: G\curvearrowright A$ such that $\alpha_g$ is approximately inner for all $g$, 
then $A\cong M_{|G|^\infty}\otimes A$.
}
\begin{proof}
Let $A^1$ and $(A_\infty)^1$ denote the unitalization of $A$ and $A_\infty$, respectively. Since $\alpha_g$ is approximately inner for every $g\in G$, we find a family of unitaries $\set{u_g}_{g\in G}\subset (A_\infty)^1$
 satisfying
\[
\alpha_g(a)=u_gau_g^*\quad \text{for all}\ a\in A.
\]
Let $\varepsilon:(A_\infty)^1\to \IC$ denote the canonical character, and define
\[
C:=\mathrm{C}^*(A,\set{u_g-\varepsilon(u_g)\ |\ g\in G})\subset A_\infty.
\]
As $C$ is separable and $\alpha$ has the Rokhlin property, one can apply \ref{rokprop} to construct a family of positive contractions
\[
\set{f_g}_{g\in G}\subset A_\infty\cap C^\prime=A_\infty\cap(A\cup\set{u_g\ |\ g\in G})^\prime
\]
such that for all $g,h,k\in G$ with $h\neq k$ and $c\in C$, the following relations hold:
\[
\Bigl( \sum_{g\in G} f_g\Bigl)c=c,\ \alpha_{g,\infty}(f_h)c = f_{gh}c,\ f_h f_k c=0,\ \text{and}\ [c,f_g]=0.
\]
Using again that $\alpha$ is pointwise approximately inner, we find a family of unitaries $\set{v_g}_{g\in G}\subset (A_\infty)^1$
with the property that
\[
\alpha_{g,\infty}(a)=v_gav_g^*\quad \text{for all}\ a\in A\cup\set{f_g\ |\ g\in G}.
\]
For $g\in G$, we define
\[
d_g:=u_g^*v_gf_1^{1/2}\in A_\infty\cap A^\prime.
\]
Note that $d_g$ indeed commutes with all elements of $A$, since this holds for both $f_1$ and $u_g^*v_g$. The latter is true as both $u_g$ and $v_g$ implement the automorphism $\alpha_g$ on $A\subset A_\infty$. Moreover, for every $a\in A$ and $g\in G$, one computes
\[
d_gd_g^*a=u_g^*v_gf_1v_g^*u_ga=u_g^* f_gu_ga=f_ga,
\]
and
\[
d_g^*d_ga=f_1^{1/2}v_g^*u_gu_g^*v_gf_1^{1/2}a=f_1 a.
\]
For every $g\in G$, we obtain well-defined elements
\[
x_g:=d_g+\ann(A,A_\infty),\ e_g:=f_g+\ann(A,A_\infty)\in F_\infty(A).
\]
The family $\set{e_g}_{g\in G}$ then defines a partition of unity consisting of projections in $F_\infty(A)$. Furthermore, for every $g\in G$, we have $x_gx_g^*=e_g$ and $x_g^*x_g=e_1$ by construction.
Using that the universal unital \cstar-algebra
\[
D:=\mathrm{C}^*\bigl(\dst y_g\ \text{partial isometry}\ |\ y_g^*y_g=y_1^*y_1,\ \sum_{g\in G}y_gy_g^*=1\dst \bigr)
\]
is isomorphic to $M_{|G|}$ via
\[
D\longrightarrow M_{|G|},\ y_g \mapsto e_{g,1},
\]
one concludes that the elements $\set{x_g}_{g\in G}$ give rise to a unital $*$-homomorphism
\[
M_{|G|}\longrightarrow F_{\infty}(A).
\]
A standard reindexation argument shows that then there exists also a unital $*$-homo\-morphism
\[
M_{|G|^\infty}\longrightarrow F_{\infty}(A).
\]
By applying either \cite[4.11]{Ki} or \cite[2.3]{TomsWinter}, we get $A\cong A\otimes M_{|G|^\infty}$.
\end{proof}

\rem We have proved the surjectivity part of \ref{ue range}, namely \ref{range}, by combining $M_{|G|^\infty}$-absorption of the given \cstar-algebra $A$ with an Elliott intertwining argument (see \ref{intertwining}), using approximate unitary equivalence in the sense of \ref{approx unitary equ} with unitaries in $\tilde{A}$. One can also apply Elliott intertwining with multiplier approximate unitary equivalence in the same manner, to carry out the constructions of \ref{psilim} and \ref{range} in the case of arbitrary separable \cstar-algebras. This way, one can prove the following weakening of \ref{ue range} for every separable \cstar-algebra $A$:

If $A\cong M_{|G|^\infty}\otimes A$, then the natural map
\[
\CR_G(A)\longrightarrow\Hom(G, \Aut(A)/_\mue),\quad [g\mapsto \alpha_g]\longmapsto [g\mapsto [\alpha_g]_{\mue}] 
\]
is surjective.

However, it remains unclear if Rokhlin actions in the sense of \ref{rokprop nawata} can exhibit rigidity like in \ref{rok ue} with respect to multiplier approximate unitary equivalence. Therefore, it is also unclear whether one can always expect a naturally induced bijection between $\quer{\CR}_G(A)$ and the corresponding $\quer{\Hom}$-set, as in \ref{ue range}.

\reme In the unital case of \ref{ue range}, it has been shown in \cite[5.17]{Phi1} that actions with the Rokhlin property are generic. However, such density results are usually not very helpful for finding examples of actions satisfying some predescribed properties, since it is not even clear how large the set of ordinary actions is. For actions on UHF-absorbing classifiable \cstar-algebras, the following consequences of \ref{ue range} should be satisfactory results in that direction:

{
\cor \label{Rok range Robert}
Let $G$ be a finite group and $A$ a \cstar-algebra with $A\cong M_{|G|^\infty}\otimes A$. Assume that $A$ is expressible as an inductive limit of 1-NCCW complexes that have trivial $K_1$-groups. For convenience, let $\Aut_+(\operatorname{Cu}^\sim(A))$ denote the set of those automorphisms $\operatorname{Cu}^\sim(A)\to\operatorname{Cu}^\sim(A)$ that send the class of a strictly positive element in $A$ to a class of a strictly positive element. 
Then the natural map
\[\CR_G(A)\longrightarrow\Hom\Bigl( G, \Aut_+\bigl( \operatorname{Cu}^\sim(A)\bigl) \Bigl),~ [g\mapsto \alpha_g]\longmapsto [g\mapsto\operatorname{Cu}^\sim(\alpha_g)] \]
is surjective and induces a bijection
\[\quer{\CR}_G(A)\longrightarrow\quer{\Hom}\Bigl( G, \Aut\bigl( \operatorname{Cu}^\sim(A) \bigl) \Bigl). \]
}
\begin{proof}
This follows directly from \ref{ue range} combined with Robert's classification theorem \cite[1.0.1]{Robert}.
\end{proof}

{
\cor \label{Rok range Kirchberg} Let $G$ be a finite group and $A$ a unital UCT Kirchberg algebra with $A\cong M_{|G|^\infty}\otimes A$. Then the natural map
\[\CR_G(A)\longrightarrow\Hom\Bigl( G, \Aut\bigl( (K_0(A),[\eins_A],K_1(A) \bigl) \Bigl),~ [g\mapsto \alpha_g]\longmapsto [g\mapsto K_*(\alpha_g)] \]
is surjective and induces a bijection
\[\quer{\CR}_G(A)\longrightarrow\quer{\Hom}\Bigl( G, \Aut\bigl( (K_0(A),[\eins_A],K_1(A) \bigl) \Bigl). \]
}
\begin{proof}
It is well-known that $*$-homomorphisms between Kirchberg algebras are classified via $KL$-theory, see \cite[Section 2.5]{Rordam}.
In particular, it is known that the natural map
\[
\Aut(A)/_\ue \to \set{ \kappa\in KL(A,A)^{-1} ~|~ \kappa_0([\eins_A]_0)=[\eins_A]_0 }
\]
is a group isomorphism. On the other hand, the UCT assumption means that the natural map
\[
KL(A,A)^{-1}\to\Aut((K_0(A),K_1(A)))
\]
is a split-surjective group homomorphism, see \cite[23.11.1]{BlaKK}. (Keep in mind that $KL$ is defined as a quotient of $KK$.) Therefore, the natural map
\[
\Aut(A)/_\ue \to \Aut\bigl( (K_0(A), [\eins_A], K_1(A) \bigl)
\]
is also split-surjective. In particular, group actions of $G$ on the Elliott invariant always lift to homomorphisms $G\to \Aut(A)/_\ue$.
Hence the surjectivity of the map
\[\CR_G(A)\longrightarrow\Hom\Bigl( G, \Aut\bigl( (K_0(A), [\eins_A], K_1(A) \bigl) \Bigl)  \]
follows directly from \ref{ue range}.
The injectivity of the map
\[\quer{\CR}_G(A)\longrightarrow\quer{\Hom}\Bigl( G, \Aut\bigl( (K_0(A),[\eins_A], K_1(A) \bigl) \Bigl) \]
follows from \cite[4.2]{Izumi2}.
\end{proof}

\rem \label{Rok range ohne UCT} It is possible to prove a result similar to \ref{Rok range Kirchberg} that does not require the assumption of the UCT. The following holds:

Let $G$ be a finite group and $A$ a unital Kirchberg algebra with $A\cong M_{|G|^\infty}\otimes A$. Let $\kappa: G\to \set{x\in KK(A,A)^{-1}\ |\ x_0([\eins]_0)=[\eins]_0}$ be a group homomorphism. Then there exists a Rokhlin action $\gamma: G\curvearrowright A$ with $KK(\gamma)=\kappa$.
\begin{proof}
A rough sketch of the proof goes as follows. Using Kirchberg-Phillips classification, one lifts $\kappa$ to a family $\set{\beta_g}_{g\in G}\subset\Aut(A)$ that defines a $G$-action on the level of $KK$. With this family of automorphisms, define the inductive system $\set{A^{(n)\oplus G}, \Phi_n}$ just as in the proof of \ref{range}. Making use of the unique $|G|$-divisibility of $KK(A,B)$ for any separable \cstar-algebra $B$ and plugging in the Milnor sequence (see \cite[21.5.2]{BlaKK}) for the functor $KK( \_\!\_\!\_ ~, B)$, one can calculate that the map
\[ 
j: A\into A_\Phi,\quad x\mapsto \Phi_{1,\infty}\Bigl( ((\beta_g)^{-1}(x))_{g\in G} \Bigl) 
\]
yields a $KK$-equivalence, and is hence asymptotically unitarily equivalent to an isomorphism $\sigma$. One defines the action $\gamma': G\curvearrowright A_\Phi$ just as in the proof of \ref{range} and pulls it back via $\sigma$, i.e. $\gamma = \sigma^{-1}\circ\gamma'\circ\sigma$. Finally, one calculates that $KK(\gamma'_g\circ j) = KK(j\circ\beta_g)$ for all $g\in G$, and hence $KK(\gamma_g)=KK(\beta_g)=\kappa(g)$ for all $g\in G$. 
\end{proof}

{
\cor \label{Rok range TAF} Let $G$ be a finite group and $A$ a separable, unital, simple and nuclear \cstar-algebra with $A\cong M_{|G|^\infty}\otimes A$. Assume that $A$ is TAF and satisfies the UCT.
Then the natural map
\[\CR_G(A)\longrightarrow\Hom\Bigl( G, \Aut\bigl( (K_0(A), K_0(A)^+,[\eins_A], K_1(A) \bigl) \Bigl)  \]
is surjective and induces a bijection
\[\quer{\CR}_G(A)\longrightarrow\quer{\Hom}\Bigl( G, \Aut\bigl( (K_0(A), K_0(A)^+, [\eins_A], K_1(A) \bigl) \Bigl). \]
}
\begin{proof}
It is known from Lin's classification of TAF algebras \cite{LinC} that $*$-homomorphisms between separable, unital, simple and nuclear TAF algebras are classified via $KL$-theory. In particular, the natural map
\[
\Aut(A)/_\ue \to \set{ \kappa\in KL(A,A)^{-1} ~|~ \kappa_0([\eins_A]_0)=[\eins_A]_0,~\kappa_0(K_0(A)^+)= K_0(A)^+ }
\]
is a group isomorphism. On the other hand, the UCT assumption means that the natural map
\[
KL(A,A)^{-1}\to\Aut((K_0(A),K_1(A)))
\]
is a split-surjective group homomorphism, see \cite[23.11.1]{BlaKK}. (Keep in mind that $KL$ is defined as a quotient of $KK$.) Therefore, the natural map
\[
\Aut(A)/_\ue \to \Aut\bigl( (K_0(A), K_0(A)^+, [\eins_A], K_1(A) \bigl)
\]
is also split-surjective. In particular, group actions of $G$ on the Elliott invariant always lift to homomorphisms $G\to \Aut(A)/_\ue$.
Hence the surjectivity of the map
\[\CR_G(A)\longrightarrow\Hom\Bigl( G, \Aut\bigl( (K_0(A), K_0(A)^+, [\eins_A], K_1(A) \bigl) \Bigl)  \]
follows directly from \ref{ue range}. Injectivity of
\[\quer{\CR}_G(A)\longrightarrow\quer{\Hom}\Bigl( G, \Aut\bigl( (K_0(A), K_0(A)^+, [\eins_A], K_1(A) \bigl) \Bigl) \]
follows from \cite[4.3]{Izumi2}.
\end{proof}

\begin{rem} 
We conjecture that the statement analogous to \ref{Rok range TAF} is true for TAI algebras, TASI algebras and most generally the recently classified \cstar-algebras of generalized tracial rank at most one, which exhaust the entire Elliott invariant for weakly unperforated $K_0$, see \cite{GongLinNiu}. 

Let $\Ell$ denote the full Elliott functor. At least the surjectivity part boils down to the question if in the UCT case, the natural map
\[
\Aut(A)/_\ue \longrightarrow \Aut(\Ell(A))
\] 
is split-surjective. In the TAF case, this is fairly clear because the Elliott invariant only involves (ordered) $K$-theory. The UCT then implies  the existence of a splitting map from the Hom-sets of $K$-theory back to $KK$-theory. For TAI algebras, this map seems to be split-sujective as well, see \cite{LinUniq}. In personal communication, Niu has confirmed to us that this map is in fact split-surjective for all separable, unital, nuclear, simple \cstar-algebras with generalized tracial rank at most one and satisfying the UCT.
\end{rem}

Next, we show how \ref{range} helps to obtain a general inductive limit decomposition for crossed products of Rokhlin actions on UHF-absorbing \cstar-algebras.

{
\theorem \label{crossed product decomposition} Let $G$ be a finite group and $A$ a separable \cstar-algebra with $A\cong M_{|G|^\infty}\otimes A$. Assume that $A$ is either unital, stable or has stable rank one. Let $\alpha: G\curvearrowright A$ be a Rokhlin action. Let $A^{(n)}=M_{|G|^{n-1}}\otimes A$ and $\alpha_g^{(n)}=\id_{M_{|G|^{n-1}}}\otimes\alpha_g$ for all $n\in\IN, g\in G$ and consider the inductive system
\[ 
\Theta_n: A^{(n+1)}\to A^{(n+2)} 
\]
defined by
\[ 
\Theta_n\Bigl( \sum_{g,f\in G} e_{g,f}\otimes x_{g,f} \Bigl) = \sum_{g,f\in G} e_{g,f}\otimes \sum_{h\in G} e_{h,h}\otimes \alpha_{h}^{(n)}(x_{gh,fh}).
\]
Then $\dst A_\Theta := \lim_{\longrightarrow} \set{ A^{(n+1)}, \Theta_n }$ is isomorphic to $A\rtimes_\alpha G$. In particular, $A\rtimes_\alpha G$ can be expressed as an inductive limit whose building blocks are isomorphic to $A$.
}
\begin{proof}
Recall the notation from \ref{range} and its proof. For $\beta_g=\alpha_g$, define the connecting maps
\[\Phi_n: A^{(n)\oplus G} \to A^{(n+1)\oplus G},\quad \Phi_n((x_g)_{g\in G})=\Bigl( \sum_{h\in G} e_{h,h}\otimes\alpha_h^{(n)}(x_{gh})\Bigl)_{g\in G}\]
and consider $\dst A_\Phi = \lim_{\longrightarrow}\set{A^{(n)\oplus G},\Phi_n}$. Combining the proof of \ref{range} with \ref{rok ue} yields that $(A,\alpha)$ is conjugate to $(A_\Phi,\gamma)$, where $\gamma: G\curvearrowright A_\Phi$ is given on each building block by
\[ \gamma^{(n)}: G \curvearrowright A^{(n)\oplus G},\quad \gamma^{(n)}_f((x_g)_{g\in G}) = (x_{f^{-1}g})_{g\in G}. \]
Hence it suffices to show that $A_\Phi\rtimes_\gamma G\cong A_\Theta$.

Let $\lambda: G\to \CU(M_{|G|})$ be the left regular representation, i.e. $\lambda(g) = \sum_{h\in G} e_{gh,h}$ for all $g\in G$. For each $n$, define
\[ \phi^{(n)}: A^{(n)\oplus G}\rtimes_{\gamma^{(n)}} G\to A^{(n+1)},\quad \sum_{g\in G} a_gu_g \mapsto \sum_{g\in G} \diag(a_g)\cdot \lambda(g). \]
It is obvious that $\phi^{(n)}$ is an isomorphism, since it is just a variant of the natural covariant representation coming from the left regular representation. Moreover, one can check that the following diagram commutes for all $n$:
\[ 
\xymatrix@C+15mm{ A^{(n)\oplus G}\rtimes_{\gamma^{(n)}} G \ar[r]^{\Phi_n\rtimes G} \ar[d]_{\phi^{(n)}} & A^{(n+1)\oplus G}\rtimes_{\gamma^{(n+1)}} G \ar[d]^{\phi^{(n+1)}} \\
A^{(n+1)} \ar[r]^{\Theta_n} & A^{(n+2)} } 
\]
Hence one gets 
\[  
A_\Phi\rtimes_\gamma G \cong \lim_{\longrightarrow} \set{ A^{(n)\oplus G}\rtimes_{\gamma^{(n)}} G, \Phi_n\rtimes G } \cong \lim_{\longrightarrow} \set{ A^{(n+1)}, \Theta_n} = A_\Theta. 
\]
\end{proof}

With the help of the above model system for Rokhlin actions on UHF-absorbing \cstar-algebras, one also observes quite easily that the crossed product is always isomorphic to the fixed point algebra.

{
\theorem \label{fixed point crossed product} Let $G$ be a finite group and $A$ a separable \cstar-algebra with $A\cong M_{|G|^\infty}\otimes A$. Assume that $A$ is either unital, stable or has stable rank one. Let $\alpha: G\curvearrowright A$ be a Rokhlin action. Then $A^\alpha\cong A\rtimes_\alpha G$.
}
\begin{proof}
Recall the notation for $A^{(n)}, \alpha^{(n)}_g, \gamma^{(n)}$ and the maps $\Phi_n: A^{(n)\oplus G}\to A^{(n+1)\oplus G},~\Theta_n: A^{(n+1)}\to A^{(n+2)}$ from \ref{range} and \ref{crossed product decomposition}. We have seen that
\[
(A,\alpha) = (A_\Phi,\gamma) = \lim_{\longrightarrow} \set{(A^{(n)\oplus G},\gamma^{(n)}),\Phi_n}
\]
and
\[ 
A\rtimes_\alpha G \cong A_\Theta = \lim_{\longrightarrow} \set{A^{(n+1)}, \Theta_n} .
\]
From the first (equivariant) inductive limit decomposition, it follows that
\[
A^\alpha \cong \lim_{\longrightarrow} \set{ (A^{(n)\oplus G})^{\gamma^{(n)}}, \Phi_n} .
\]
The building blocks are obviously just isomorphic copies of $A^{(n)}$ as every element $(x_g)_g\in A^{(n)\oplus G}$, which is fixed by the shift action $\gamma^{(n)}$, has the same entry everywhere. The connecting map $\Phi_n$ then restricts to
\[ 
\Omega_n : A^{(n)}\to A^{(n+1)},\quad \Omega_n(x) = \sum_{h\in G} e_{h,h}\otimes\alpha_h^{(n)}(x). 
\]
Let $\sigma$ be the flip automorphism on $M_{|G|}\otimes M_{|G|}$. Recall also the right-regular representation $\rho: G\to\CU(M_{|G|})$ given by $\rho(g)=\sum_{h\in G} e_{hg^{-1},h}$. Then we see for every $n\in\IN$ and $x=(x_{g,f})_{g,h\in G}\in A^{(n+1)}$ that
\[ 
\begin{array}{ccl}
\multicolumn{3}{l}{ \Theta_n(x) } \\\\
&=& \dst\sum_{g,f\in G} e_{g,f}\otimes \sum_{h\in G} e_{h,h}\otimes\alpha_h^{(n)}(x_{gh,fh}) \\\\
&=& \dst(\sigma\otimes\id_{A^{(n)}})\left( \sum_{h\in G} e_{h,h}\otimes \sum_{g,f\in G} e_{g,f}\otimes \alpha_h^{(n)}(x_{gh,fh}) \right) \\\\
&=& \dst(\sigma\otimes\id_{A^{(n)}})\left( \sum_{h\in G} e_{h,h}\otimes\alpha_h^{(n+1)}\Bigl( \sum_{g,f\in G} e_{gh^{-1},fh^{-1}}\otimes x_{g,f} \Bigl) \right) \\\\
&=& \dst(\sigma\otimes\id_{A^{(n)}})\left( \sum_{h\in G} e_{h,h}\otimes\alpha_h^{(n+1)}\circ\ad\bigl( \rho(h)\otimes\eins_{\CM(A^{(n)})} \bigl)\Bigl( \sum_{g,f\in G} e_{g,f}\otimes x_{g,f} \Bigl) \right) \\\\
&=& \dst(\sigma\otimes\id_{A^{(n)}})\circ\ad\Bigl(\sum_{h\in G} e_{h,h}\otimes\rho(h)\otimes\eins_{\CM(A^{(n)})}\Bigl)\circ\Omega_{n+1}(x)
.\end{array}
\]
Now the flip automorphism $\sigma$ is well-known to be inner. The unitary implementing it is
\[ \sum_{g,h\in G} e_{g,h}\otimes e_{h,g}\quad\in M_{|G|}\otimes M_{|G|}. \]
In particular, $\Theta_n$ is unitarily equivalent to $\Omega_{n+1}$ for all $n$.
By \ref{intertwining}(2), this yields an isomorphism between the inductive limits, so in particular between $A^\alpha$ and $A\rtimes_\alpha G$.
\end{proof}


\section{Local representability of duals of Rokhlin actions on UHF-absorbing \cstar-algebras}
\noindent
In this section, we look at actions of finite abelian groups. Particularly, we take a look at the class of Rokhlin actions treated in the previous section, and will describe the dual actions in terms of local representability.

\rem If $B\subset A$ is a non-degenerate sub-\cstar-algebra, then one can see easily for every $b,c\in A_\infty\cap B'$ that $b+\ann(B, A_\infty)=c+\ann(B, A_\infty)$ implies $ba=ca$ and $ab=ac$ for all $a\in A$. Thus it makes sense to multiply elements of $F_\infty(B,A)$ with elements of $A$ to obtain elements in $A_\infty$. For example,
\[
A\longrightarrow A_\infty,\quad a\mapsto cac^*
\]
yields a well-defined c.p.c.~map for all contractions $c\in F_\infty(B,A)$. If $c$ happens to be a unitary in $F_\infty(B,A)$, then this defines a $*$-homomorphism, whose restriction to $B$ is just the canonical inclusion $B\subset A\into A_\infty$.

\begin{defi}[cf.~{\cite[Section 3.2]{Izumi}} and {\cite[Section 4]{Nawata}}] \label{lr}
Let $A$ be a separable \cstar-algebra, $G$ a finite abelian group and $\alpha: G\curvearrowright A$ an action. 
\begin{enumerate}
\item $\alpha$ is called approximately representable, if there is a unitary representation $w: G\to F_\infty(A^\alpha,A^\alpha)\subset F_\infty(A^\alpha,A)$ with $\alpha_g(a) = w(g)aw(g)^*$ for all $a\in A$ and $g\in G$.
\item $\alpha$ is called locally representable, if there exists a sequence of \cstar-subalgebras $A_n\subset A$ with $A_n\subset A_{n+1}$ and $A=\quer{\bigcup_n A_n}$ and unitary representations $w_n: G\to\CU(\CM(A_n))$ such that
$$\alpha_g|_{A_n} = \ad(w_n(g))\quad\text{for all}\; g\in G\;\text{and}\; n\in\IN.$$
\item Let $\FC$ be a class of separable \cstar-algebras. $\alpha$ is called locally $\FC$-representable, if it is locally representable and the $A_n$ from (2) may be chosen to be isomorphic to a \cstar-algebra in $\FC$.
\end{enumerate}
\end{defi}

Recall the following duality between Rokhlin actions and approximately representable actions:

\begin{theorem}[see {\cite[3.8]{Izumi}} and {\cite[4.4]{Nawata}}] \label{duality}
Let $\alpha$ be an action of a finite abelian group $G$ on a separable \cstar-algebra $A$. Then
\begin{itemize}
\item[(i)] $\alpha$ has the Rokhlin property iff $\hat{\alpha}$ is approximately representable.
\item[(ii)] $\alpha$ is approximately representable iff $\hat{\alpha}$ has the Rokhlin property.
\end{itemize}
\end{theorem}

For Rokhlin actions on UHF-absorbing \cstar-algebras, we can partially strengthen the assertion about its dual action:

{
\theorem \label{local rep} Let $G$ be a finite abelian group and $A$ a separable \cstar-algebra with $A\cong M_{|G|^\infty}\otimes A$. Assume that $A$ is either unital, stable, or has stable rank one. Let $\alpha: G\curvearrowright A$ be a Rokhlin action. Then $\hat{\alpha}: \widehat{G}\curvearrowright A\rtimes_\gamma G$ is locally $\set{A}$-representable.
}
\begin{proof}
Recall the notation from \ref{range} and \ref{crossed product decomposition}. In particular, the definition of the connecting maps
\[
\Phi_n: A^{(n)\oplus G} \to A^{(n+1)\oplus G}\quad\text{and}\quad \Theta_n: A^{(n+1)}\to A^{(n+2)}.
\]
Combining the proof of \ref{range} with \ref{rok ue} yields that $(A,\alpha)$ is conjugate to $(A_\Phi,\gamma)$, where $\gamma: G\curvearrowright A_\Phi$ is given on each building block by
\[
\gamma^{(n)}: G \curvearrowright A^{(n)\oplus G},\quad \gamma^{(n)}_f((x_g)_{g\in G}) = (x_{f^{-1}g})_{g\in G}. \]
Moreover, recall the isomorphisms 
\[
\phi^{(n)}: A^{(n)\oplus G}\rtimes_{\gamma^{(n)}} G\to A^{(n+1)},\quad \sum_{g\in G} a_gu_g \mapsto \sum_{g\in G} \diag(a_g)\cdot S(g).
\]
that made the following diagram commutative for all $n$:
\[
\xymatrix@C+15mm{ A^{(n)\oplus G}\rtimes_{\gamma^{(n)}} G \ar[r]^{\Phi_n\rtimes G} \ar[d]_{\phi^{(n)}} & A^{(n+1)\oplus G}\rtimes_{\gamma^{(n+1)}} G \ar[d]^{\phi^{(n+1)}} \\
A^{(n+1)} \ar[r]^{\Theta_n} & A^{(n+2)} }
\]
For each $n$, we define the action $\delta^{(n)}: \widehat{G}\curvearrowright A^{(n+1)}$ induced by the unitary representation
\[ 
\widehat{G}\to \CU(\CM(A^{(n+1)})), \quad \chi\mapsto U(\chi)\otimes\eins_{\CM(A^{(n)})}\quad\text{with}\quad U(\chi)=\sum_{h\in G} \chi(h)\cdot e_{h,h}. 
\]
As one easily checks that
\[ 
U(\chi) \lambda(g) U(\chi)^* = \chi(g)\cdot \lambda(g)\quad\text{for all}~g\in G,~\chi\in\widehat{G},
\]
one gets that the diagram
\[ 
\xymatrix@C+15mm{ A^{(n)\oplus G}\rtimes_{\gamma^{(n)}} G \ar[r]^{\widehat{\gamma^{(n)}}_\chi } \ar[d]_{\phi^{(n)}} & A^{(n)\oplus G}\rtimes_{\gamma^{(n)}} G \ar[d]^{\phi^{(n)}} \\
A^{(n+1)} \ar[r]^{\delta^{(n)}_\chi} & A^{(n+1)} } 
\]
commutes for all $n$ and $\chi\in\widehat{G}$. In particular, the building block actions $\delta^{(n)}$ give rise to an action $\delta: \widehat{G}\curvearrowright A_\Theta$.

The two commuting diagrams above yield equivariant inductive limit decompositions
\[
\begin{array}{ccl}
(A_\Phi\rtimes_\gamma G, \hat{\gamma}) 
&\cong& \dst\lim_{\longrightarrow} \set{ (A^{(n)\oplus G}\rtimes_{\gamma^{(n)}} G, \widehat{\gamma^{(n)}}), \Phi_n\rtimes G } \vspace{1mm}\\
&\cong& \dst\lim_{\longrightarrow} \set{ (A^{(n+1)},\delta^{(n)}), \Theta_n }\vspace{1mm}\\
&\cong& (A_\Theta, \delta). 
\end{array}
\]
Since $(A,\alpha)\cong (A_\Phi,\gamma)$, the action $\hat{\alpha}$ is conjugate to $\delta$. As $\delta$ is clearly locally $\set{A}$-representable, this finishes the proof.
\end{proof}


\section{Examples and applications}
\noindent
In this section, we give some applications to the results that have been proved in the previous sections.

Recall the following useful result due to Izumi, which simplifies the computation of the $K$-theory of crossed products by finite group actions with the Rokhlin property:

\begin{theorem}[see {\cite[3.13]{Izumi}}] \label{Rok K-theory}
Let $A$ be a simple unital \cstar-algebra, $G$ a finite group and $\alpha: G\curvearrowright A$ a Rokhlin action. Then
\[K_i(A^\alpha) = \bigcap_{g\in G} \ker(\id-K_i(\alpha_g))\quad\text{for}~i=0,1 \]
and the inclusion $A^\alpha\into A$, on the level of $K$-theory, coincides with the canonical inclusion of this subgroup.
\end{theorem}

The following recovers and generalizes \cite[6.3.4]{Blackadar} by combining \ref{Rok range TAF} with known classification results.

\example Let $(G_0,G_0^+,u)$ be a countable, uniquely 2-divisible scaled ordered abelian group, which is weakly unperforated and has the Riesz interpolation property. Let $\sigma$ be a scaled ordered group automorphism of order 2 on $G_0$, such that $\ker(\id-\sigma)$ is isomorphic to $(\IZ[\frac{1}{2}],\IZ[\frac{1}{2}]^+,1)$ as an ordered group with order unit.
Let $G_1$ be any countable, uniquely 2-divisible abelian group. Then there exists a $\IZ_2$-action $\gamma$ on $M_{2^\infty}$ such that $M_{2^\infty}^\gamma$ is a simple UCT TAF \cstar-algebra with $K_0(M_{2^\infty}^\gamma)\cong G_0$ as ordered groups and $K_1(M_{2^\infty}^\gamma)\cong G_1$. Moreover, $\gamma$ is locally $\set{M_{2^\infty}^\gamma}$-representable.
\begin{proof}
By combining the range result \cite{EllGong} with \cite{LinAH}, there exists a separable, unital, nuclear, simple \cstar-algebra $A$ with tracial rank zero and satisfying the UCT such that
\[ \bigl( K_0(A), K_0(A)^+, [\eins_A], K_1(A) \bigl) = \bigl( G_0,G_0^+,u, G_1 \bigl). \] 
We make use of the classification theorem of \cite{LinC}. As the $K$-groups of $A$ are uniquely 2-divisible, we have $A\cong M_{2^\infty}\otimes A$.

Consider the $\IZ_2$-action on the above quadruple of $A$ whose $K_0$-part agrees with $\sigma$, and the $K_1$-part acts as multiplication by $-1$. By \ref{Rok range TAF}, we can lift this to a Rokhlin action $\alpha: \IZ_2\curvearrowright A$ with $K_0(\alpha)=\sigma$ and $K_1(\alpha)=-\id$. Applying \ref{Rok K-theory} yields that the Elliott invariant of $A\rtimes_\alpha\IZ_2$ is the same as that of the CAR algebra $M_{2^\infty}$. As taking crossed products with Rokhlin actions preserves the property of being TAF (see \cite[2.6]{Phi}) and satisfying the UCT (see \cite[Section 5]{Santiago}), $A\rtimes_\alpha\IZ_2$ is again TAF and satisfies the UCT, and hence $A\rtimes_\alpha\IZ_2\cong M_{2^\infty}$ by \cite{LinC}.
Under this identification, the dual action $\hat{\alpha}$ gives rise to a locally $\set{A}$-representable (see \ref{local rep}) action $\gamma: \IZ_2\curvearrowright M_{2^\infty}$ with
\[ M_{2^\infty}^\gamma \cong (A\rtimes_\alpha\IZ_2)^{\hat{\alpha}} \cong A. \]
\end{proof}

\reme An analogous statement is true for actions of any finite abelian group. We state it in the next example for $\IZ_p$ instead of $\IZ_2$. As the proof is very similar to the above, we omit it.

\begin{example} \label{general Blackadar} 
Let $p\geq 2$ be a natural number. Let $(G_0,G_0^+,u)$ be a countable, uniquely $p$-divisible ordered abelian group with order unit, which is weakly unperforated and has the Riesz interpolation property. Let $\sigma_0$ be an ordered group automorphism of order $p$ on $G_0$, such that $\ker(\id-\sigma_0)$ is isomorphic to $(\IZ[\frac{1}{p}],\IZ[\frac{1}{p}]^+,1)$ as an ordered group with order unit.
Let $G_1$ be a countable, uniquely $p$-divisible abelian group with an order $p$ automorphism over $\sigma_1$ such that $\ker(\id-\sigma_1)=0$. Then there exists a $\IZ_p$-action $\gamma$ on $M_{p^\infty}$ such that $M_{p^\infty}^\gamma$ is a simple UCT TAF \cstar-algebra with $K_0(M_{p^\infty}^\gamma)\cong G_0$ as ordered groups and $K_1(M_{p^\infty}^\gamma)\cong G_1$. Moreover, $\gamma$ can be chosen to be locally $\set{M_{p^\infty}^\gamma}$-representable.
\end{example}

Next, we shift our focus in the direction of UHF-absorbing (UCT) Kirchberg algebras. As it turns out, the range results of \cite[4.8(3), 4.9]{Izumi} and \cite[6.4]{Izumi2}
about approximately representable actions of cyclic groups of prime power cardinality on $\CO_2$ can be recovered and extended by combining \ref{Rok range Kirchberg} with Kirchberg-Phillips classification. (We note, however, that in the following more general cases, this does not necessarily cover the whole range.)
Moreover, just as for \ref{general Blackadar}, an analogous result holds for all finite abelian groups, but we only state it for the cyclic case:

\example \label{general Izumi} Let $p\geq 2$ be a natural number. For $i=0,1$, let $G_i$ be a countable, uniquely $p$-divisible abelian group with an order $p$ automorphism $\sigma_i\in\Aut(G_i)$ satisfying $\ker(\id-\sigma_i)=0$. Then there exists an action $\gamma: \IZ_p\curvearrowright\CO_2$ such that $\CO_2^\gamma$ is a unital UCT Kirchberg algebra with
\[
K_i(\CO_2^\gamma)\cong K_i(\CO_2\rtimes_\gamma\IZ_p)\cong G_i.
\]
Moreover, $\gamma$ is locally $\set{\CO_2^\gamma}$-representable.
The automorphisms $\sigma_i$ correspond to $K_i(\hat{\gamma})$ under the above $K$-theory isomorphism.
\begin{proof}
Let $A$ be a unital UCT Kirchberg algebra in Cuntz standard form with $K_i(A)\cong G_i$ for $i=0,1$. As the $G_i$ are uniquely $p$-divisible, we have $A\cong A\otimes M_{p^\infty}$. By \ref{Rok range Kirchberg}, we can find a Rokhlin action $\alpha: \IZ_p\curvearrowright A$ such that $K_i(\alpha)=\sigma_i$. By \ref{Rok K-theory}, we obtain $K_i(A\rtimes_\alpha\IZ_p)=0$. Moreover, since $\alpha$ has the Rokhlin property, $A\rtimes_\alpha\IZ_p$ is a unital UCT Kirchberg algebra. This implies that $\CO_2\cong A\rtimes_\alpha\IZ_p$. Under this identification, the dual action $\hat{\alpha}$ gives rise to a locally $\set{A}$-representable (see \ref{local rep}) action $\gamma: \IZ_p\curvearrowright \CO_2$ with
\[ \CO_2^\gamma \cong (A\rtimes_\alpha\IZ_p)^{\hat{\alpha}} \cong A \]
and
\[ \CO_2\rtimes_\gamma\IZ_p \cong (A\rtimes_\alpha\IZ_p)\rtimes_{\hat{\alpha}}\IZ_p \cong M_p\otimes A \cong A .\]
Moreover, $K_i(\hat{\gamma})$ corresponds to $K_i(\hat{\hat{\alpha}})=K_i(\alpha)$, which corresponds to $\sigma_i$.
\end{proof}

\example \label{Izumi action}
A special case of the above is $G_0=\IZ[\frac{1}{2}]$ with $\sigma_0=-\id$ and $G_1=0$. This yields a locally UCT Kirchberg-representable action $\gamma$ on $\CO_2$ with $\CO_2\rtimes_\gamma\IZ_2\cong M_{2^\infty}\otimes\Ost$. This action actually coincides with the action constructed in \cite[4.7]{Izumi}.

\reme Example \ref{general Izumi} can be generalized further, as one may delete the UCT assumption with the help of \ref{Rok range ohne UCT} to cover an (at least a priori) greater class of possible fixed point algebras. As this is not entirely obvious, we sketch how this can be done.

\rem \label{crossed product KK} Let $G$ be a finite group and let $A$ be a separable, nuclear \cstar-algebra with $A\cong M_{|G|^\infty}\otimes A$. Assume that $A$ is either unital, stable or has stable rank one. Let $\gamma: G\curvearrowright A$ be a Rokhlin action and let $\kappa: G\to KK(A,A)^{-1},~\kappa(g)=KK(\gamma_g)$ be the induced homomorphism to $KK$. For every separable \cstar-algebra $B$, we have an isomorphism
\[
KK(A\rtimes_\gamma G, B) \cong \set{ x\in KK(A,B)~|~ \kappa(g)\otimes x = x~\text{for all}~g\in G}
\]
of abelian groups.
\begin{proof}
We omit the most technical details. The key to proving this is \ref{crossed product decomposition}. If one plugs in the Milnor sequence (see \cite[21.5.2]{BlaKK}) for the functor $KK( \_\!\_\!\_ ~, B)$, one can describe the abelian group $KK(A\rtimes_\gamma G, B)$ as the stationary inverse limit
\[
KK(A\rtimes_\gamma G, B) \cong \lim_{\longleftarrow} \set{ KK(A,B), \psi},
\]
where $\psi(x) = \sum_{g\in G} \kappa(g)\otimes x$. For this argument, one also has to use unique $|G|$-divisibility to see that the image of $\psi^2$ is equal to the image of $\psi$ (since $\psi^2=|G|\cdot\psi$), in order to deduce that $\dst{\lim_{\longleftarrow}}^1\set{KK(A,B), \psi}$ vanishes for all $B$.
Using unique $|G|$-divisibility of the group $KK(A,B)$ one more time, one can see that this limit yields the group in the assertion.
\end{proof}

\example 
Let $p\geq 2$ be a natural number.
Let $A$ be a unital Kirchberg algebra in Cuntz standard form with $A\cong M_{p^\infty}\otimes A$.  Assume that $\kappa\in KK(A,A)$ is an element satisfying the equations $1=\kappa^p$ and $0 = \sum_{j=0}^{p-1}\kappa^j$. Then there exists a locally $\set{A}$-representable action $\gamma: \IZ_p\curvearrowright\CO_2$ such that $\CO_2^\gamma\cong \CO_2\rtimes_\gamma\IZ_p\cong A$ and such that $KK(\hat{\gamma})=\kappa$.
\begin{proof}
We use \ref{Rok range ohne UCT} to find a Rokhlin action $\alpha: \IZ_p\curvearrowright A$ with $KK(\alpha)=\kappa$. The previous Remark \ref{crossed product KK} allows one to compute the $KK$-groups for the crossed product $A\rtimes_\alpha\IZ_p$. We have 
\[ 
KK( A\rtimes_\alpha\IZ_p, B)\cong\set{ x\in KK(A,B)~|~ \kappa\otimes x = x}
\]
as abelian groups for every separable \cstar-algebra $B$.
 If one keeps in mind the second equality for $\kappa$ from above, one obtains
\[ x = p\cdot \frac{x}{p} =  \Bigl( \sum_{j=0}^{p-1} \kappa^j \Bigl)\cdot \frac{x}{p} = 0 \]
for all $x\in KK(A\rtimes_\alpha\IZ_p, A\rtimes_\alpha\IZ_p)$. Hence we have $A\rtimes_\alpha\IZ_p\cong\CO_2$. The action $\gamma$ that corresponds to $\hat{\alpha}$ under this identification, is locally $\set{A}$-representable by \ref{local rep}. We have
\[\CO_2^\gamma \cong \bigl( A\rtimes_\alpha\IZ_p \bigl)^{\hat{\alpha}} \cong A \]
and
\[ \CO_2\rtimes_\gamma\IZ_p \cong (A\rtimes_\alpha\IZ_p)\rtimes_{\hat{\alpha}}\IZ_p \cong M_p\otimes A \cong A.\]
Moreover, we see that $KK(\hat{\gamma})=KK(\hat{\hat{\alpha}})=KK(\alpha)=\kappa$.
\end{proof}

\defi Let $\FC_R$ be the class of \cstar-algebras that are expressible as an inductive limit of 1-dimensional NCCW complexes with trivial $K_1$-groups. Note that \cite{Robert} provides a classification result for $\FC_R$. In particular, \cite[6.2.4]{Robert} reduces the classifying invariant in the simple case to the $K_0$-group, the tracial simplex and the pairing between these.

\rem[compare to {\cite[5.6]{Nawata}}] 
A result of Elliott \cite[5.2.1 and 5.2.2]{Ell} shows that there exists a simple, stably projectionless \cstar-algebra $B$ in $\FC_R$ with 
\[ \bigl( K_0(B), (T(B), T_1(B)), r_B \bigl) = \bigl( \IZ, (\IR^+, \emptyset), 0 \bigl). \]
By \cite[5.2]{Nawata2}, there exists a hereditary subalgebra $\Bst$ of $B$ such that $\Bst$ has a unique tracial state $\tau$ and no unbounded traces. Since $\Bst$ is stably isomorphic to $B$ by Brown’s theorem, we have
\[ \bigl( K_0(\Bst), (T(\Bst), T_1(\Bst)), r_{\Bst} \bigl) = \bigl( \IZ, (\IR^+, \set{\tau}), 0 \bigl). \]
By Robert's classification theorem \cite{Robert}, $\Bst$ is (up to isomorphism) the unique simple, stably projectionless \cstar-algebra in $\FC_R$ with this data.

\rem The notation $\Bst$ is justified, as $\Bst$ should be thought of a stably projectionless analogue of $\Ost$. A unital Kirchberg algebra $A$ is in Cuntz standard form, i.e. $[\eins_A]_0$ is trivial in $K_0(A)$, if and only if $A\cong A\otimes\Ost$. Similarly, one might be tempted to conjecture that a separable, simple, nuclear, stably projectionless, $\CZ$-stable \cstar-algebra $A$ has a trivial pairing map $0=r_A: T(A)\to S(K_0(A))$ if and only if $A\cong A\otimes\Bst$.

\begin{rem} \label{Rok K-theory UHF}
As we need it in the next example, we remark that a similar formula as in \ref{Rok K-theory} holds for the non-unital (and non-simple) case as well. To be more precise, we have the following:

Let $G$ be a finite group and $A$ a separable \cstar-algebra with $A\cong M_{|G|^\infty}\otimes A$. Let $\alpha: G\curvearrowright A$ be a Rokhlin action. Then
\[ 
K_i(A^\alpha)\cong K_i(A\rtimes_\alpha G) \cong \bigcap_{g\in G} \ker(\id-K_i(\alpha_g))\quad\text{for}~i=0,1. 
\]
As this follows easily from \ref{crossed product decomposition} and \ref{fixed point crossed product} after stabilizing with the trivial action on the compacts $\CK$, we omit the proof.
\end{rem}


The following example is essentially the same as the one treated in \cite[5.6]{Nawata}.

\example \label{Nawata action}
Let $\CW$ denote the so-called Razak-Jacelon algebra from \cite{Jacelon}. There exists an action $\gamma: \IZ_2\curvearrowright\CW$ with 
\[
\CW^\gamma\cong\CW\rtimes_\gamma\IZ_2\cong M_{2^\infty}\otimes\Bst
\]
and such that $\gamma$ is locally $\CW^\gamma$-representable.
\begin{proof}
By Robert's classification theorem, there exists an automorphism $\beta\in\Aut(M_{2^\infty}\otimes\Bst)$ with $K_0(\beta)=-\id$. As $\beta^2$ is trivial on the Elliott invariant, $\beta^2$ is approximately inner. Note that all \cstar-algebras in $\FC_R$ have stable rank one. Hence, we can apply \ref{ue range} and get a Rokhlin action $\alpha: \IZ_2\curvearrowright M_{2^\infty}\otimes\Bst$ with $K_0(\alpha)=-\id$. As $\FC_R$ is closed under taking crossed products with Rokhlin actions of finite groups (see \cite[Section 5]{Santiago}), $(M_{2^\infty}\otimes\Bst)\rtimes_\alpha\IZ_2$ is in $\FC_R$ and has trivial $K$-theory by \ref{Rok K-theory UHF}. Moreover, it has a unique tracial state and no unbounded traces. Hence it is isomorphic to the Razak-Jacelon algebra $\CW$. Under this identification, the dual action $\hat{\alpha}$ gives rise to a locally $\set{M_{2^\infty}\otimes\Bst}$-representable (see \ref{local rep}) action $\gamma: \IZ_2\curvearrowright\CW$ with
\[
\CW^\gamma \cong \bigl( (M_{2^\infty}\otimes\Bst)\rtimes_\alpha\IZ_2 \bigl)^{\hat{\alpha}} \cong M_{2^\infty}\otimes\Bst 
\]
and
\[ 
\CW\rtimes_\gamma\IZ_2 \cong ((M_{2^\infty}\otimes\Bst)\rtimes_\alpha\IZ_2)\rtimes_{\hat{\alpha}}\IZ_2 \cong M_2\otimes M_{2^\infty}\otimes\Bst \cong M_{2^\infty}\otimes\Bst .
\]
\end{proof}

\noindent
At last, we would like to show that one can reduce the UCT problem for separable, nuclear \cstar-algebras to a question about certain locally representable actions on $\CO_2$.

\begin{rem}[see {\cite[10.2]{Neukirch}}] \label{integer ring}
Let $p\geq 2$ be a natural number. Let us pick a primitive $p$-th root of unity $\xi_p = \exp(2\pi i/p)\in\IC$. Then the ring generated by $\IZ$ and $\xi_p$, written $\IZ[\xi_p]$, coincides with the ring of integers in the number field $\IQ(\xi_p)$. The additive subgroup of this ring is well-known to be free abelian, with rank equal to $[\IQ(\xi_p):\IQ]$, which coincides with the value of Euler's phi-function at $p$
\[
\phi(p) = \bigl|\set{j\in\set{1,\dots,p} ~|~ \operatorname{gcd}(j,p)=1}\bigl|. 
\]
For example, if $p$ happens to be prime, then $\phi(p)=p-1$.
Hence, we have $\IZ[\xi_p]\cong \IZ^{\phi(p)}$ as abelian groups.
\end{rem}

The following contains \ref{Izumi action} as a special case and is simultaneously contained in \ref{general Izumi} as a special case.

{
\prop \label{minimal O2 actions}
Let $p\geq 2$ be a natural number. Then there exists a locally UCT Kirchberg-representable action $\gamma_p: \IZ_p\curvearrowright\CO_2$ such that $A_p = \CO_2\rtimes_{\gamma_p}\IZ_p$ is $KK$-equivalent to $M_{p^\infty}^{\oplus\phi(p)}$.
}
\begin{proof}
Choose a unital UCT Kirchberg algebra $A_p$ with $K$-theory
\[
(K_0(A_p), [\eins_{A_p}]_0, K_1(A)) \cong (\IZ[\mbox{$\frac{1}{p}$}]^{\oplus\phi(p)}, 0, 0).
\]
By the UCT, $A_p$ is in fact $KK$-equivalent to $M_{p^\infty}^{\oplus\phi(p)}$, since they have identical $K$-theory. Because of \ref{integer ring}, $K_0(A_p)$ is isomorphic to the additive group of the ring $\IZ[\frac{1}{p}, \xi_p]$. Under this identification, we obtain an order $p$ automorphism $\sigma: K_0(A_p)\to K_0(A_p)$ by $x\mapsto\xi_p\cdot x$. Note that obviously $\ker(\id-\sigma)=0$.

Since the $K$-theory of $A_p$ is uniquely $p$-divisible, we have $A_p\cong M_{p^\infty}\otimes A_p$. By \ref{Rok range Kirchberg}, there exists a Rokhlin action $\alpha: \IZ_p\curvearrowright A_p$ with $K_0(\alpha)=\sigma$. Note that from the above observations, combined with \ref{Rok K-theory} and Kirchberg-Phillips classification, it follows that $A_p\rtimes_\alpha\IZ_p\cong\CO_2$. Under this identification, the dual action $\gamma_p=\hat{\alpha}: \IZ_p\curvearrowright\CO_2$ yields a locally $\set{A_p}$-representable (see \ref{local rep}) action with $\CO_2\rtimes_{\gamma_p}\IZ_p\cong A_p$. This finishes the proof.
\end{proof}

The following is well-known:

{
\prop \label{UCT direct sum}
Let $n\in\IN$ be a natural number and $A_1,\dots, A_n$ separable \cstar-algebras. Then each $A_i$ satisfies the UCT if and only if $A_1\oplus\dots\oplus A_n$ satisfies the UCT.
}

{
\prop \label{UCT UHF}
Let $A$ be a separable \cstar-algebra, and let $p,q\geq 2$ be two relatively prime natural numbers. If both $M_{p^\infty}\otimes A$ and $M_{q^\infty}\otimes A$ satisfy the UCT, then so does $A$.
}
\begin{proof}
One has that the \cstar-algebra $Z_{p^\infty,q^\infty}\otimes A$ arises as an extension of $M_{p^\infty}\otimes A\oplus M_{q^\infty}\otimes A$ by $\CC_0(0,1)\otimes M_{p^\infty}\otimes M_{q^\infty}\otimes A$, which both satisfy the UCT by assumption. Hence also $Z_{p^\infty,q^\infty}\otimes A$ satisfies the UCT. Since the Jiang-Su algebra arises as a stationary inductive limit (with injective connecting maps) of $Z_{p^\infty,q^\infty}$, we have that also $\CZ\otimes A \sim_{KK} A$ satisfies the UCT.
\end{proof}

{
\theorem \label{UCT}
Let $p,q\geq 2$ be two prime numbers. The following are equivalent:
\begin{enumerate}
\item Every separable, nuclear \cstar-algebras satisfies the UCT.
\item Every unital Kirchberg algebra satisfies the UCT.
\item If $\beta: \IZ_p\curvearrowright\CO_2$ and $\gamma: \IZ_q\curvearrowright\CO_2$ are pointwise outer, locally Kirchberg-representable actions, then both $\CO_2\rtimes_\beta\IZ_p$ and $\CO_2\rtimes_\gamma\IZ_q$ satisfy the UCT.
\item If $\gamma: \IZ_{pq}\curvearrowright\CO_2$ is a pointwise outer, locally Kirchberg-representable action, then $\CO_2\rtimes_\gamma\IZ_{pq}$ satisfies the UCT.
\end{enumerate}
}
\begin{proof}
First we observe that the implications $(1)\implies (2),(3),(4)$ are trivial. We will first show the implications $(3)\implies (2)$ and $(4)\implies (2)$:

$(3)\implies (2)$: Assume that (2) is false. Then we can pick a unital Kirchberg algebra $A$ that does not satisfy the UCT. By \ref{UCT direct sum} and \ref{UCT UHF}, it follows that either $A\otimes M_{p^\infty}^{\oplus (p-1)}$ or $A\otimes M_{q^\infty}^{\oplus (q-1)}$ does not satisfy the UCT.

Then it follows from \ref{minimal O2 actions} actions that either
\[
(A\otimes\CO_2)\rtimes_{\id_A\otimes\gamma_p} \IZ_p \cong A\otimes A_p \sim_{KK} A\otimes M_{p^\infty}^{\oplus (p-1)}
\]
or
\[
(A\otimes\CO_2)\rtimes_{\id_A\otimes\gamma_q} \IZ_q \cong A\otimes A_q \sim_{KK} A\otimes M_{q^\infty}^{\oplus (q-1)}
\]
does not satisfy the UCT. Since both $\gamma_p$ and $\gamma_q$ are pointwise outer and locally Kirchberg-representable and $A\otimes\CO_2\cong\CO_2$, this gives a counterexample to (3).

$(4)\implies (2)$: Assume that (2) is false. Then we can pick a unital Kirchberg algebra $A$ that does not satisfy the UCT. By \ref{UCT direct sum} and \ref{UCT UHF}, it follows that either $(A\otimes M_{p^\infty})^{\oplus (p-1)\cdot q}$ or $(A\otimes M_{q^\infty})^{\oplus (q-1)\cdot p}$ does not satisfy the UCT. Now by \cite{GoldIz}, there exist two pointwise outer, locally $\set{\CO_\infty}$-representable actions $\alpha_p: \IZ_p\curvearrowright\CO_\infty$ and $\alpha_q: \IZ_q\curvearrowright\CO_\infty$ with
\[
\CO_\infty\rtimes_{\alpha_p}\IZ_p \sim_{KK} C^*(\IZ_p) \cong \IC^p\quad\text{and}\quad \CO_\infty\rtimes_{\alpha_q}\IZ_q \sim_{KK} C^*(\IZ_q)\cong\IC^q.
\]
Recalling that $\IZ_{pq}\cong\IZ_p\oplus\IZ_q$, we obtain two pointwise outer, locally Kirchberg-representable actions $\gamma_p\otimes\alpha_q: \IZ_{pq}\curvearrowright\CO_2\otimes\CO_\infty\cong\CO_2$ and $\gamma_q\otimes\alpha_p: \IZ_{pq}\curvearrowright\CO_2\otimes\CO_\infty\cong\CO_2$ with
\[
\CO_2\rtimes_{\gamma_p\otimes\alpha_q}\IZ_{pq}\cong A_p\otimes (\CO_\infty\rtimes_{\alpha_q}\IZ_q) \sim_{KK} M_{p^\infty}^{\oplus (p-1)\cdot q}
\]
and
\[
\CO_2\rtimes_{\gamma_q\otimes\alpha_p}\IZ_{pq}\cong A_q\otimes (\CO_\infty\rtimes_{\alpha_p}\IZ_p) \sim_{KK} M_{q^\infty}^{\oplus (q-1)\cdot p}.
\]
In particular, it follows that either
\[
(A\otimes\CO_2)\rtimes_{\id_A\otimes\gamma_p\otimes\alpha_q}\IZ_{pq} \sim_{KK} (A\otimes M_{p^\infty})^{\oplus (p-1)\cdot q}
\]
or
\[
(A\otimes\CO_2)\rtimes_{\gamma_q\otimes\alpha_p}\IZ_{pq} \sim_{KK} (A\otimes M_{q^\infty})^{\oplus (q-1)\cdot p}
\]
does not satisfy the UCT. Hence this gives a counterexample to (4).

Finally, let us show the implication $(2)\implies (1)$. We stress that this is already well-known and due to Kirchberg (see \cite[2.17]{Ki} for an even stronger assertion), but we repeat the argument for the reader's convencience:

Let $A$ be a separable, nuclear \cstar-algebra. By adding a unit, if necessary, we can assume that $A$ is unital. (This procedure does not change if $A$ is in the UCT class or not.) We show that $A$ is $KK$-equivalent to a unital Kirchberg algebra, which immediately confirms the implication $(2)\implies (1)$.

For this, we moreover assume that $A\cong A\otimes\Ost$. If this is not the case, we can stabilize with $\Ost$, since this does not change the $KK$-equivalence class. If $A\cong A\otimes\Ost$, there exists a unital embedding $\iota: \CO_2\to A$. Let us explicitely choose two isometries $s_1,s_2\in A$ with $\eins=s_1s_1^*+s_2s_2^*$. By Kirchberg's embedding theorem (see \cite{KiPhi}), there exists a unital embedding $\kappa: A\to\CO_2$. Now define an endomorphism
\[
\phi: A\to A\quad\text{via}\quad \phi(x)=s_1xs_1^*+s_2(\iota\circ\kappa)(x)s_2^*.
\]
Consider the inductive limit $B=\dst\lim_{\longrightarrow}\set{A, \phi}$. Then $B$ is clearly again separable, unital, nuclear and $\CO_\infty$-absorbing. Since for all $x\neq 0$, the element
\[
s_2^*\phi(x)s_2 = \iota\circ\kappa(x)
\]
is the image of the full element $\kappa(x)\in\CO_2$, it follows that $\phi(x)$ is also full. Hence $B$ is simple.

Lastly, it is immediate that $KK(\phi)=1+KK(\iota\circ\kappa)=1$, since $\iota\circ\kappa$ factors through $\CO_2$. In particular, the connecting maps of this inductive system induce $KK$-equivalences. Hence it follows that the canonical embedding $\phi_\infty: A\to B$ induces a $KK$-equivalence. This is implied by \cite[2.4]{Dad}, which basically boils down to plugging in the Milnor sequence \cite[21.5.2]{BlaKK} for the functor $KK(\_\!\_\!\_~,B)$ in this situation.  

To summarize, we have found a unital Kirchberg algebra $B$ that is $KK$-equivalent to $A$. This finishes the proof.
\end{proof}

\rem From the proof of \ref{UCT}, we can say a bit more in the case that we consider actions of only one natural number. Namely, we get the following:

Let $p\geq 2$ be a natural number. Assume that for all pointwise outer, locally Kirchberg-representable actions $\gamma: \IZ_p\curvearrowright\CO_2$, the crossed product $\CO_2\rtimes_\gamma\IZ_p$ satisfies the UCT. Then all separable, nuclear and $M_{p^\infty}$-absorbing \cstar-algebras satisfy the UCT.


\end{document}